\documentclass[12pt,centertags,oneside]{amsart}
\usepackage[T1]{fontenc}
\usepackage[utf8]{inputenc}
\usepackage{amsmath}
\usepackage{amsthm}
\usepackage{amssymb}
\usepackage[mathscr]{eucal}
\usepackage{bbm}
\usepackage{pgf,tikz}
\usepackage[bookmarks,colorlinks=true,linkcolor=cyan,citecolor=blue,urlcolor=cyan]{hyperref}
\usepackage{faktor}
\usepackage{enumitem}
\usepackage{verbatim}
\usepackage{pgfplots}
\usepackage{marginnote}
\usepackage{fullpage}
\usepackage{color}

\usetikzlibrary{arrows}
\usetikzlibrary{shapes.geometric}
\usetikzlibrary{cd}
\usetikzlibrary{intersections, pgfplots.fillbetween}

\newcounter{basicblock}[section]
\renewcommand{\thebasicblock}{\thesection.\arabic{basicblock}}
\newenvironment{block}[1][]
{\refstepcounter{basicblock}\par\medskip\noindent\textbf{#1~\thebasicblock}.}
{\medskip}
\newenvironment{block*}[1][] {\par\medskip\noindent\textbf{#1}} {\medskip}

\newcommand{\ep}{\varepsilon}

\let\oldforall\forall
\renewcommand{\forall}{\oldforall \,}
\let\oldexists\exists
\renewcommand{\exists}{\oldexists \,}
\let\emptyset\varnothing

\newcommand{\dist}{\mathrm{dist}}

\newcommand{\Res}{\mathrm{Res}\,}
\newcommand{\lp}{\langle}
\newcommand{\rp}{\rangle}

\newcommand{\ddc}{\mathrm{dd^c}}
\newcommand{\de}{\partial}
\newcommand{\debar}{\bar{\partial}}

\newcommand{\cC}{\mathcal{C}}
\newcommand{\cD}{\mathcal{D}}

\newcommand{\cK}{\mathcal{K}}

\newcommand{\cU}{\mathcal{U}}

\newcommand{\C}{\mathbb{C}}
\newcommand{\D}{\mathbb{D}}

\newcommand{\N}{\mathbb{N}}
\newcommand{\Pb}{\mathbb{P}}
\newcommand{\R}{\mathbb{R}}
\newcommand{\Sb}{\mathbb{S}}
\newcommand{\Z}{\mathbb{Z}}

\setlist[itemize]{noitemsep, topsep=3pt}
\setlist[enumerate]{noitemsep, topsep=3pt}
\newenvironment{nlist}
    {\begin{enumerate}[label=(\roman*)]}
    {\end{enumerate}}

\newenvironment{defn}{\begin{block}[Definition]}{\end{block}}
\newenvironment{lm}{\begin{block}[Lemma]\begin{itshape}}{\end{itshape}\end{block}}
\newenvironment{prop}{\begin{block}[Proposition]\begin{itshape}}{\end{itshape}\end{block}}
\newenvironment{thm}{\begin{block}[Theorem]\begin{itshape}}{\end{itshape}\end{block}}
\newenvironment{cor}{\begin{block}[Corollary]\begin{itshape}}{\end{itshape}\end{block}}
\newenvironment{ex}{\begin{block}[Example]}{\end{block}}
\newenvironment{rmk}{\begin{block}[Remark]}{\end{block}}

\usepackage{amsaddr}

\title{Green currents of holomorphic correspondences on compact K\"ahler manifolds}
\author{Muhan Luo$\,^\text{a}$ and Marco Vergamini$\,^\text{b}$}
\address{$^\text{a}\,$National University of Singapore, Lower Kent Ridge Road 10, Singapore 119076, Singapore}
\address{$^\text{b}\,$Scuola Normale Superiore, Piazza dei Cavalieri 7, 56126 Pisa, Italy}
\email{e0708207@u.nus.edu \textnormal{(Muhan Luo)}}
\author{}
\email{marco.vergamini@sns.it \textnormal{(Marco Vergamini)}}
\date{}

\pgfplotsset{compat=1.18}

\begin{document}

\begin{abstract}
Consider a holomorphic correspondence $f$ on a compact K\"ahler manifold $X$ of dimension $k$. Let $1\le q\le k$ be any integer such that the dynamical degrees of $f$ satisfy $d_{q-1}<d_q$. We construct the Green currents $T_c$ of $f$ associated with the classes $c$ belonging to the dominant eigenspace for the action of $f^*$ on $H^{q,q}(X,\R)$. We also show that the super-potential of $T_c$ is $\log$-H\"older-continuous. When $f$ has simple action on cohomology and its graph satisfies an assumption on the local multiplicity, we prove the exponential equidistribution of all positive closed currents towards the main Green current, i.e., the only one associated to the unique maximal degree $d_q$. 
\end{abstract}

\maketitle

\medskip


\smallskip


\tableofcontents

\section{Introduction} \label{intro}

\emph{Green currents} are among the central objects in the study of a holomorphic dynamical system $f$. They are $f$-invariant positive closed currents, obtained as limit currents for the dynamics of $f$. The Green currents and their regularity properties have been studied extensively in many settings. See \cite{B65AM,FLM83BSBM,L83ETDS} for the classical case of dimension $1$ (in this case, the Green currents are measures). In higher dimension, see, for instance, \cite{DS09AM,S99PS} for the case of endomorphisms of projective spaces, \cite{BS91IM,BS92MA,S99PS} for the case of complex Hénon maps, and \cite{C01AM,D-TD12AM,DS05JAMS,DS10JAG} for the case of automorphisms of compact K\"ahler manifolds.

\emph{Holomorphic correspondences} (roughly speaking, multivalued holomorphic maps; see Section \ref{sec: pullback} for a precise definition) naturally appear in a number of settings, see for instance \cite{CU03JRAM,M02NTM}. From a dynamical point of view, they contain the complexity of both endomorphisms of projective spaces and automorphisms of K\"ahler manifolds. Their study has been developed in \cite{AV24AMP,DNV18AM,DS06CMH}, see also \cite{DKW21PAMQ} for an application to the dynamics of random matrices. In this paper, we systematically study the properties of Green currents associated with holomorphic correspondences. As was the case for the other dynamical systems mentioned above, this is a necessary first step towards the understanding of more refined dynamical properties.

\smallskip

The first goal of this paper is to construct the Green currents of a holomorphic correspondence $f$ and to study their regularity. This provides a generalization of the main results of \cite{DS05JAMS} to the case of correspondences, see also \cite{DS10JAG}. Recall that the \textit{dynamical degree of order $q$} of $f$, denoted by $d_q$, is the spectral radius of the pullback $f^*$ of $f$ acting on $H^{q,q}(X,\R)$, see Section \ref{sec: pullback}. We denote by $F$ (resp. $H$) the real dominant (resp. strictly dominant) subspace of $H^{q,q}(X,\R)$ for the action of $f^*$, see Section \ref{linear-alg}. When a $T_c$ as in the statement is non-zero, we say it is a \textit{Green current of order $q$} of $f$. 

\begin{thm}\label{main-thm-1}
    Let $f$ be a holomorphic correspondence on a compact K\"ahler manifold $X$ of dimension $k$. Suppose $1\leq q\leq k$ is an integer such that $d_{q-1}<d_q$, and let $\cD_q$ be the real space generated by positive closed $(q,q)$-currents.
    \begin{enumerate}
        \item Each class $c\in F$ has a representative $T_c\in\cD_q$  which depends linearly on $c$ and satisfies $f^*(T_c)=T_{f^*(c)}$. In particular, when $c\in H$, we have $f^*(T_c)=d_qT_c$.
        \item The super-potential of $T_c$ is $\log^r$-continuous for some $r>0$ and all $c\in F$.
    \end{enumerate}
\end{thm}

To prove Theorem \ref{main-thm-1}, we will construct the Green currents by giving an explicit formula for their \emph{super-potentials}. The $\log$-H\"older-continuity of the super-potentials follows from studying the regularity of the super-potential of $f^*(\Omega)$ for a closed $(q,q)$-form $\Omega$. This is done by applying a version of \L{}ojasiewicz type inequality. The main difficulties come from the non-invertibility of $f$ and the existence of both critical points and critical values.

Super-potentials were introduced by Dinh-Sibony in \cite{DS09AM} as a generalization of potentials for $(1,1)$-currents. The crucial point is that they are functions defined on a suitable set of exact currents. Like for classical potentials, their regularity properties imply strong regularity properties for the corresponding currents and have provided powerful tools in handling delicate problems in high-dimensional complex dynamics. See for instance \cite{BD26AM,dTV10MSMF,DNV18AM,DS10JAG}.

Being $\log$-H\"older-continuous (see Definition \ref{defn: regularity-SP}) is a condition which is “exponentially” weaker than being H\"older-continuous. The recent works of Bianchi-Dinh \cite{BD23JMPA,BD24GAFA} have shown the importance of such functions in the study of equilibrium states for projective spaces. The $\log$-H\"older regularity of the equilibrium measures for meromorphic correspondences was proved in \cite{DW23AMP}. More recently, to study mixing properties of endomorphisms of projective spaces and automorphisms of compact Kähler manifolds, the second author \cite{V25Ax} extended the notion of $\log$-H\"older-continuous regularity to currents of arbitrary bidegree.

\medskip
 
Our second result concerns the dynamical properties of a holomorphic correspondence with simple action on cohomology. We say the action of $f$ on cohomology is \textit{simple} if there exists an integer $0\leq q\leq k$ such that $d_q$ is strictly larger than all the other dynamical degrees and the action of $f^*$ on $H^{q,q}(X,\C)$ has only one eigenvalue of maximal modulus which is simple and equal to $d_q$. We call $d_q$ the {\it main dynamical degree} of $f$. Denote by $f^{-1}$ the \emph{adjoint correspondence} of $f$, see Section \ref{sec: pullback}. If $f^{-1}$ is holomorphic, it also has simple action on cohomology and $d_{k-q}(f^{-1})=d_q(f)=d_q$ is the main dynamical degree. The unique Green current of mass 1 associated with $f$ (resp. $f^{-1}$) is denoted by $T^+$ (resp. $T^-$).

The dynamics of holomorphic automorphisms (i.e.\ the special case when $d_0=d_k=1$) with simple action on cohomology has been studied extensively, see e.g.\ \cite{A24Ax,DS10CM,DS10JAG}. One of the main tools is the equidistribution towards the Green currents. Several equidistribution results have also been proved for endomorphisms of projective spaces (which are a special case of $d_0=1$ when $X=\Pb^k$), see e.g.\ \cite{A16TAMS,A18MA,DS09AM,T11AM}. For correspondences, it is shown in \cite{DNV18AM} that for any real smooth closed $(q,q)$-form $\alpha$, $d_q^{-n}(f^n)^*(\alpha)$ converges to a multiple of $T^+$. In \cite{AV24AMP}, this result is extended to non-pluripolar products of positive closed $(1,1)$-currents. Let $\Gamma_{f^n}$ be the graph of $f^n$ in $X\times X$. This object has been used to study some statistical properties of the dynamics, see e.g.\ \cite{BD26AM,VW25Ax}. It is shown in \cite{DNV18AM} that $d_q^{-n}[\Gamma_{f^n}]$ converges to $T^+\otimes T^-$. 
 
Under some mild assumptions on the local multiplicities of the graph of a holomorphic correspondence, (an analogue of not having highly critical periodic orbits in dimension $1$), our second main result strengthens the results of \cite{DNV18AM}. See also \cite{Luo25PAMS} for the case of dimension 1 and \cite[Corollary 4.8]{D05BSMF} for polynomial correspondences. More precisely, we show that the convergence towards Green currents is exponentially fast when tested against smooth forms. In the case of automorphisms of K\"ahler manifolds, this property has been the key to establish many strong statistical properties of the dynamical system, see e.g.\ \cite{BD26AM,DS10CM,VW25Ax}. We also show that these assumptions are satisfied by generic holomorphic correspondences. In particular, this is true for polynomial correspondences up to finite iterations (see Corollaries \ref{pk-corr} \& \ref{poly-small}). The definitions of $q$-small adjoint multiplicity and small multiplicity are given in Definition \ref{def:inv_small_mult}.

\begin{thm}\label{main-thm-2}
    Let $f$ be a holomorphic correspondence on $X$ such that $f^{-1}$ is also holomorphic. Suppose $f$ has simple action on cohomology and that the main dynamical degree is given by $d_q=d_q(f)$. Let $S$ be a current in $\cD_q$ such that $\lp S,T^-\rp=1$. 
    \begin{enumerate}
        \item If $f$ has $q$-small adjoint multiplicity, then $d_q^{-n}(f^{n})^*(S)$ converges to $T^+$ exponentially fast;
        \item If $f$ has small multiplicity, then the sequence of positive closed $(k,k)$-currents $d_q^{-n}[\Gamma_{f^n}]$ converges to $T^+\otimes T^-$ exponentially fast.
    \end{enumerate}
\end{thm}

The key ingredients of our proof of (1) are a Skoda-type inequality along with a precise estimate of the regularity of super-potentials of smooth exact forms after pushforward. (2) is then a direct consequence of (1) by applying it to the system $F:=(f,f^{-1})$. Precise statements can be found in Section \ref{sec:multi-equidis}.

\medskip

\noindent\textbf{Acknowledgements.} The first author has received funding from the National University of Singapore and MOE of Singapore through the grant A-8002488-00-00. A part of this work was done during the first author's visit to the Centro di Ricerca Matematica Ennio De Giorgi and the University of Pisa. He would like to thank them for their hospitality and the excellent working conditions.

\medskip

\noindent\textbf{Statements and Declarations.} The authors have no conflicts of interest to declare that are relevant to the content of this article.

\section{Preliminaries}\label{sec:pre}

\subsection{Super-potentials of currents}
Let $(X,\omega)$ be a compact K\"ahler manifold of dimension $k$. Fix an integer $0\leq q\leq k$. We denote by $\cC_q$ the convex cone of positive closed $(q,q)$-currents on $X$, by $\cD_q$ the real space generated by $\cC_q$ and by $\cD_q^0$ the subspace of currents in $\cD_q$ whose cohomology classes are $0$ in $H^{q,q}(X,\R)$. We use $\{\cdot\}$ to denote the cohomology class of a closed current.

\begin{defn}
    For any $S\in\cD_q$, the {\it $*$-norm} of $S$ is defined by
\[
    \|S\|_*:=\min \|S^+\|+\|S^-\|,
\]
where the minimum is taken over all $S^+$ and $S^-$ in $\cC_q$ such that $S=S^+-S^-$.

Let $S$ and $(S_n)$ be currents in $\cD_q$. We say that $S_n$ \emph{$*$-converge} to $S$ if $S_n$ converges to $S$ in the sense of currents and the $S_n$'s are uniformly $*$-bounded. We call the topology induced by $*$-convergence the \emph{$*$-topology}. Smooth forms are dense in $\cD_q$ and $\cD_q^0$ for the $*$-topology, see \cite[Theorem 2.4.4]{DS10JAG}.
\end{defn}

Since $X$ is compact, we can fix a finite atlas. For every $l>0$, we consider the standard $\|\cdot\|_{\mathcal{C}^l}$ norm of forms associated to that atlas. Different choices of atlases give equivalent norms. Therefore, we can fix an atlas, and hereafter every dependence on it will be omitted.

By duality, we can consider the following norms on $\mathcal{D}_q$.

\begin{defn}
    Given a current $S$ in $\mathcal{D}_q$ and $l>0$, we define
    $$\|S\|_{-l}:=\sup\{|\langle S,\Omega\rangle|\mid \Omega\text{ is a smooth }(k-q,k-q)\text{-form with }\|\Omega\|_{\mathcal{C}^l}\le1\}.$$

    The norm $\|\cdot\|_{-l}$ induces a distance $\dist_l$ given by $\dist_l(S,S'):=\|S-S'\|_{-l}$.
\end{defn}

The following result is obtained using a standard interpolation between Banach spaces \cite{T78Book}, see for instance \cite[Proposition 2.2.1]{DS10JAG}.

\begin{prop}\label{interpol}
    Let $l$ and $l'$ be real numbers with $0<l<l'$. Then, on any $\|\cdot\|_*$-bounded subset of $\mathcal{D}_q$, the topologies induced by $\dist_l$ and by $\dist_{l'}$ coincide with the weak topology and the $*$-topology. Moreover, for every $\|\cdot\|_*$-bounded subset of $\mathcal{D}_q$, there is a constant $c_{l,l'} > 0$ such that
    $$\dist_{l'} \le \dist_l \le c_{l,l'}(\dist_{l'})^{l/l'}.$$
\end{prop}
\vspace*{-\baselineskip}

Now we recall some general results about super-potentials of currents on $X$. For more details, see e.g.\ \cite{BD26AM,DNV18AM,DS10JAG}. Let $\alpha=(\alpha_1,\dots,\alpha_h)$ be a fixed family of real smooth closed $(q,q)$-forms such that $\{\alpha\}=(\{\alpha_1\},\dots,\{\alpha_h\})$ is a basis of $H^{q,q}(X,\R)$. By Poincaré duality, we can also find a family of real smooth closed $(k-q,k-q)$-forms $(\check{\alpha}_1,\dots,\check{\alpha}_h)$ such that $(\{\check{\alpha}_1\},\dots,\{\check{\alpha}_h\})$ is the dual basis of $\{\alpha\}$ with respect to the cup product. For any $R\in\cD^0_{k-q+1}$, we can find a $(k-q,k-q)$-current $U_R$ such that $\ddc U_R=R$ and $\lp U_R,\alpha_i\rp=0$ for all $i$. We can choose it smooth if $R$ is smooth. We call such a $U_R$ the {\it $\alpha$-normalized potential} of $R$.

\begin{defn} \label{def_su-po}
    For a current $S\in\cD_q$, its {\it $\alpha$-normalized super-potential} is the linear functional on smooth forms $R$ in $\cD^0_{k-q+1}$ defined by
    \[
        \cU_S(R):=\lp S,U_R\rp,
    \]
    where $U_R$ is a $\alpha$-normalized potential of $R$.
\end{defn}

The above definition depends on the choice of $\alpha$ but is independent of the choice of $U_R$ as long as it is $\alpha$-normalized. In the rest of this paper, we fix the family $\alpha$. When we mention the super-potential of a current, we always assume it is $\alpha$-normalized. 
    
\begin{rmk}
    In \cite{DNV18AM}, the super-potential of $S$ is defined as follows: fix a smooth positive closed $(q,q)$-form $\beta$ which is cohomologous to $S$. Then $\cU_S(R):=\lp S-\beta,U_R\rp$ which depends on the choice of $\beta$. It is easy to check that if we take $\beta=\sum_{i=1}^h\lp S,\check{\alpha}_i\rp\alpha_i$, these two definitions coincide. 
\end{rmk}

We can consider convergence of currents in terms of their super-potentials. 

\begin{defn}
    Let $S$ and $(S_n)$ be currents in $\cD_q$ with  $\cU_S$ and $\cU_{S_n}$ being their super-potentials, respectively. We say that $S_n$ converges {\it SP-uniformly} to $S$ if $S_n\to S$ in the sense of currents and  $\cU_{S_n}$ converges uniformly to $\cU_S$ on any $*$-bounded set of smooth forms in $\cD_{k-q+1}^0$.
\end{defn}

We are interested in the regularity of the super-potential $\cU_S$ of $S\in\cD_q$ as a function on $\cD^0_{k-q+1}$ with respect to the $*$-topology and the $\|\cdot\|_{-l}$ norm.

\begin{defn}\label{defn: regularity-SP}
   Take a current $S\in\mathcal{D}_q$ with super-potential $\cU_S$ and positive constants $C$, $L$, $\eta$ and $r$. We say that
    \begin{enumerate}
        \item $S$ has a \textit{continuous super-potential} if $\mathcal{U}_S$ extends to a linear functional defined on all of $\mathcal{D}_{k-q+1}^0$ which is continuous with respect to the $*$-topology.
        
        \item $\mathcal{U}_S$ is \emph{$(C,\eta)$-H\"older-continuous} if it is continuous and we have $$|\mathcal{U}_S(R)|\le C\|R\|_{-2}^\eta\qquad\text{for every }R\in\mathcal{D}^0_{k-q+1}\text{ with }\|R\|_*\le 1.$$

        We say that $\cU_S$ is \emph{$\eta$-H\"older-continuous} if it is $(C,\eta)$-H\"older-continuous for some $C>0$.

        \item $\mathcal{U}_S$ is \emph{$(L,r)$-$\log$-H\"older-continuous} if it is continuous and we have $$|\mathcal{U}_S(R)|\le \frac{L}{(1+|\log{\|R\|_{-2}}|)^r}\qquad\text{ for every }R\in\mathcal{D}^0_{k-q+1}\text{ with }\|R\|_*\le 1.$$

        We say that $\cU_S$ is \emph{$\log^r$-continuous} if it is $(L,r)$-$\log$-H\"older-continuous for some $L>0$.
    \end{enumerate}
\end{defn}

Take $S_1\in \cD_{q_1}$ and $S_2\in\cD_{q_2}$ with $q_1+q_2\le k$. If $S_1$ has a continuous super-potential, the wedge product $S_1\wedge S_2$ is well-defined. This is a generalization of the classical theory of Bedford-Taylor \cite{BT82Acta}. By \cite[Proposition 3.4.2]{DS10JAG}, when $S_1$ and $S_2$ both have H\"older-continuous super-potentials, so does $S_1\wedge S_2$. Moreover, let $S$ and $S'$ be positive currents in $\cD_q$ such that $S'\le S$. Classical domination principles extend to super-potentials, see e.g.\ \cite[Theorem 1.1]{DNV18AM}. Here, we will need the following results, see \cite[Propositions 3.10 and 3.11]{V25Ax}.

\begin{lm} \label{DS10JAG-3_4_2}
    Take $S_1\in\mathcal{D}_{q_1}$ and $S_2\in\mathcal{D}_{q_2}$ with $q_1+q_2\le k$. Suppose $S_j$ has a $\eta_j$-H\"older-continuous super-potential for some $\eta_j>0$, $j=1,2$. Then $S_1\wedge S_2$ has a $\eta$-H\"older-continuous super-potential, where $\eta=\min\{\eta_1,1\}\cdot\min\{\eta_2,1\}/2$.
\end{lm}

\begin{lm} \label{DNV18AM-1_1}
    Let $S$ and $S'$ be currents in $\cC_q$ for some $1\le q\le k$. Assume that $S'\le S$. If $S$ has a $\eta$-H\"older-continuous super-potential for some $0<\eta\le1$, then $S'$ has a $\big(C',\eta/(50k)\big)$-H\"older-continuous super-potential. The positive constant $C'$ depends on $S$, but is independent of $S'$.
\end{lm}

Finally, we have the following Skoda-type estimate from \cite{DS10CM}. For every integer $0\le s\le k$, we denote by $\cC_s^c$, $\cD_s^c$ and $\cD_s^{0c}$ the subset of currents in, respectively, $\cC_s$, $\cD_s$ and $\cD_s^0$ with continuous super-potentials. By \cite[Remark 4.5]{DNV18AM}, for a current $T\in\cD_q^0$, the action of its super-potential $\cU_T$ can be extended to $\cD_{k-q+1}^{0c}$ by $\cU_T(R):=\cU_R(T)$ for any $R\in\cD_{k-q+1}^{0c}$.

\begin{prop}\label{skoda}
    Let $R$ be a current in $\cD_{k-q+1}^0$ with $\|R\|_*\leq 1$ whose super-potential $\cU_R$ is $(C,\eta)$-H\"older-continuous. Then there exists a constant $A>0$, independent of $R$, $\eta$ and $C$, such that the super-potential $\cU_S$ of $S$ satisfies 
        \[
            |\cU_S(R)|\leq A(1+\eta^{-1}\log^+C)
        \]
    for any $S\in\cD_q^0$ with $\|S\|_*\leq 1$, where $\log^+x:=\max\{\log{x},0\}$ for every $x>0$.
\end{prop}

\subsection{Holomorphic correspondences}\label{sec: pullback}

Let $\pi_1$ and $\pi_2$ be the canonical projection maps from $X\times X$ to its two factors. Consider an effective $k$-cycle $\Gamma=\sum_j\Gamma_j$ which is a finite sum of irreducible analytic sets $\Gamma_j$ of dimension $k$ in $X\times X$. We assume that $\pi_1(\Gamma_j)=\pi_2(\Gamma_j)=X$ for all $j$. The {\it meromorphic correspondence} $f$ associated to $\Gamma$ is given by 
\[
    f(A):=\pi_2(\pi_1^{-1}(A)\cap\Gamma)
\]
for any set $A\subset X$. We say that $\Gamma$ is the {\it graph} of $f$. We say $f$ is a {\it holomorphic correspondence} if the {\it indeterminacy set} $I(f):=\{x\in X: \dim\pi_1^{-1}(x)\cap \Gamma>0\}$ is empty. In this paper, we only consider holomorphic correspondences. The {\it adjoint correspondence} is denoted by $f^{-1}$. Its graph is given by the image of $\Gamma$ under the involution $(x,y)\mapsto (y,x)$. Note that $f^{-1}$ may not be holomorphic, even when $f$ is. For two holomorphic correspondences $f$ and $g$ on $X$, we define their composition $f\circ g$ simply by $f\circ g(x)=f(g(x))$, counting multiplicity, for any $x\in X$. This is still a holomorphic correspondence. In particular, we can consider the $n$-th iterate $f^n$ of $f$, which is the composition of $f$ with itself for $n$ times.

\medskip

We now recall some results about the pullback and pushforward operators associated with a holomorphic correspondence. For their proofs, we refer to \cite{DNV18AM} for the case of pushforward, and to Appendix \ref{appendix} for the case of pullback. Let $f$ be a holomorphic correspondence on $X$ with graph $\Gamma$. Its pullback and pushforward actions on a current $T$ are defined by
\[
    f^*(T):=(\pi_1)_*(\pi_2^*(T)\wedge[\Gamma])\qquad\text{and}\qquad f_*(T):=(\pi_2)_*(\pi_1^*(T)\wedge[\Gamma])
\]
whenever the wedge product is meaningful. In particular, they are well-defined when $T$ is smooth. The definition of $f_*$ can be extended continuously, with respect to the convergence of currents, to a linear operator from $\cD_q$ to itself for $0\leq q\leq k$. It preserves the cone $\cC_q$ of positive closed $(q,q)$-currents and $\cD_q^0$. Thus, it defines a linear map $f_*$ on the cohomology group $H^{q,q}(X,\R)$. For every $T\in \cD_q$, we have $f_*\{T\}=\{f_*(T)\}$, and $(f^n)_*=(f_*)^n$ on both $\cD_q$ and $H^{q,q}(X,\R)$.

The definition of $f^*$ can be extended to a linear operator from $\cD_q^c$ to itself for $0\leq q\leq k$. This extension is continuous in the following sense: if a sequence $T_n$ converges SP-uniformly to $T$, then the sequence $f^*(T_n)$ converges to $f^*(T)$ in the sense of currents. We have that $f^*$ preserves $\cC_q^c$ and $\cD_q^{0c}$. Thus, it defines a linear map $f^*$ on the cohomology group $H^{q,q}(X,\R)$. For every $T\in \cD_q^c$, we have $f^*\{T\}=\{f^*(T)\}$, and $(f^n)^*=(f^*)^n$ on both $\cD_q^c$ and $H^{q,q}(X,\R)$. Moreover, the action of $f^*$ on $H^{q,q}(X,\R)$ is dual to the one of $f_*$ on $H^{k-q,k-q}(X,\R)$.

\medskip

The \textit{dynamical degree of order $q$} of $f$, denoted by $d_q(f)$, is the spectral radius of $f^*$ on $H^{q,q}(X,\R)$. Since $f^*$ preserves the cone of positive closed currents, by Perron-Frobenius theorem $d_q(f)$ is an eigenvalue for the action of $f^*$ on $H^{q,q}(X,\R)$. We will simply write $d_q$ when there is no ambiguity. In particular, $d_0$ is the number of points of $f(x)$ for any $x\in X$, counted with multiplicity, and $d_k=d_k(f)=d_0(f^{-1})$ is the number of points of $f^{-1}(x)$ for a generic $x\in X$.

Since $\omega^q$ is strictly positive, we have $\|(f^n)^*(\omega^q)\|\sim n^{m-1}d_q^n$ for some $m\ge1$. This comes from the definition of the dynamical degrees and some linear algebra facts, see Section \ref{linear-alg} and \cite[Lemma 4.1.2]{DS10JAG}. As a consequence, we get the following lemma.

 \begin{lm}\label{lemma:dyna-degree}
     Let $f$ be as above. For $0\leq q\leq k$, we have 
     \[
         d_q(f)=\lim_{n\to\infty}\left[\int_X(f^n)^*\omega^q\wedge\omega^{k-q}\right]^{1/n}=\lim_{n\to\infty}\left[\int_X\omega^q\wedge (f^n)_*\omega^{k-q}\right]^{1/n}.
     \]
 \end{lm}

When $f$ is an automorphism, the inequalities of Khovanskii-Teissier-Gromov \cite{G90ADGT,K79UMN,T79CRASP} imply that the function $s\mapsto \log d_s$ is concave. This implies the following \emph{monotonicity} condition: there are integers $0\leq p\leq p'\leq k$ such that
\[
    1=d_0<d_1<\cdots<d_p=\cdots=d_{p'}>\cdots>d_{k-1}>d_k=1.
\]
Log-concavity can be extended to the case of holomorphic correspondences provided that there exists a subsequence $(n_i)_{i\in\N}$ such that the graph of $f^{n_i}$ is irreducible for all $i\in\N$. This condition is expected to be true for a generic irreducible correspondence on $X$. However, log-concavity fails for general holomorphic correspondences on compact K\"ahler manifolds. See \cite[Theorem 1.1 \& Remark 1.5]{TTT20Crelle}.

It is then natural to ask if monotonicity still holds. This property is proven in \cite{BDR24TAMS} for H\'enon-like and polynomial-like maps, but the following example shows that it fails for general holomorphic correspondences.

\begin{ex} \label{not_mono}
    For each $s\geq 1$, let $f_s$ be a holomorphic endomorphism of $\Pb^k$ of algebraic degree $s$. It is known that $H^{q,q}(\Pb^k,\R)\cong\R$ and the action $f_s^*$ on $H^{q,q}(\Pb^k,\R)$ is just the multiplication by $s^q$. Hence the dynamical degrees satisfy $d_q(f_s)=s^q$ for $0\leq q\leq k$. By \cite[Lemma 4.7]{DNV18AM}, $f_s^{-1}$ is a holomorphic correspondence and, by duality, $d_q(f_s^{-1})=s^{k-q}$. Let $\tilde{f}$ be the holomorphic correspondence on $\Pb^k$ whose graph is given by the sum of the graphs of $f_{s_1}$ and $f_{s_2}^{-1}$. It is easy to see that we have $d_q(\tilde{f})=s_1^q+s_2^{k-q}$. It can be computed that, for many choices of $(s_1,s_2)$, the sequence $\{s_1^q+s_2^{k-q}\}_{q=0}^k$ is decreasing when $q$ is small and increasing when $k-q$ is small. 
\end{ex}

Let $S$ be a positive closed $(q,q)$-current with a continuous super-potential. Then $f^*(S)$ also has a continuous super-potential. Moreover, the next lemma gives an explicit characterization of $f^*(S)$. For the proof, see \cite[Lemma 4.6]{DNV18AM}.

\begin{lm} \label{DNV18AM_4-6}
    Let $\Omega$ be a Zariski dense open set in $X$ such that the restriction $\tau_1$ (respectively, $\tau_2$) of $\pi_1$ (respectively,  $\pi_2$) to $\Gamma\cap\pi_1^{-1}(\Omega)$ is an unramiﬁed covering (respectively, an unramiﬁed map). Let $S$ be a positive closed $(q,q)$-current on $X$ with a continuous super-potential. Then $f^*(S)$ is the extension by $0$ of $(\tau_1)_*\tau_2^*(S)$ to $X$.
\end{lm}

Let $\omega_1$ and $\omega_2$ be two smooth positive forms. In general, we do not have $f^*(\omega_1\wedge\omega_2)=f^*(\omega_1)\wedge f^*(\omega_2)$ when $f$ is a correspondence. However, a one-sided inequality is true.

\begin{lm} \label{wpb}
    Let $\omega_1$ and $\omega_2$ be two smooth positive closed forms on $X$. Then we have
    \begin{equation} \label{wedge_pullback}
    f^*(\omega_1\wedge\omega_2) \le f^*(\omega_1)\wedge f^*(\omega_2)
    \end{equation}
    in the sense of currents.
\end{lm}

\begin{proof}
    Let $\Omega,\tau_1,\tau_2$ be as in Lemma \ref{DNV18AM_4-6}. Since $\omega_1$ and $\omega_2$ are smooth, they have continuous super-potentials, and $\omega_1\wedge\omega_2$ has a continuous super-potential too. By \cite[Lemma 4.2]{DNV18AM}, $f^*(\omega_1\wedge\omega_2) $, $ f^*(\omega_1) $ and $ f^*(\omega_2)$ have no mass on $X\setminus\Omega$. So we only need to check \eqref{wedge_pullback} on $\Omega$.

    Take a point $x\in\Omega$. Since $\tau_1$ is an unramified covering, there exist an open neighbourhood $U\subseteq\Omega$ of $x$ and open subsets $V_1,\dots,V_d\subseteq\Gamma\cap\pi_1^{-1}(\Omega)$ such that $\tau_1\vert_{V_j}:V_j\longrightarrow U$ is a biholomorphism. For every $1\le j\le d$, take $\psi_j:=(\tau_1\vert_{V_j})^{-1}$. On $U$ we have
    \begin{align*}
        f^*(\omega_1\wedge\omega_2)&=(\tau_1)_*\tau_2^*(\omega_1\wedge\omega_2)=\sum_{j=1}^d\psi_j^*\tau_2^*(\omega_1\wedge\omega_2)=\sum_{j=1}^d\psi_j^*\tau_2^*(\omega_1)\wedge\psi_j^*\tau_2^*(\omega_2)\\
        &\le \sum_{j=1}^d\sum_{h=1}^d\psi_j^*\tau_2^*(\omega_1)\wedge\psi_h^*\tau_2^*(\omega_2)=\left(\sum_{j=1}^d\psi_j^*\tau_2^*(\omega_1)\right)\wedge\left(\sum_{h=1}^d\psi_h^*\tau_2^*(\omega_2)\right)\\
        &=(\tau_1)_*\tau_2^*(\omega_1)\wedge(\tau_1)_*\tau_2^*(\omega_2)=f^*(\omega_1)\wedge f^*(\omega_2),
    \end{align*}
    as desired. The proof is complete.
\end{proof}

We also have the following version of \cite[Proposition 5.8]{D05JGA} for correspondences. 

\begin{prop} \label{D05JGA-5_7-corr}
    There exists $C_0>0$ such that for every $n\ge1$ the norm $A_{p,q,n}$ of $(f^n)^*$ on $H^{p,q}(X,\C)$ satisfies
    $$A^2_{p,q,n}\le C_0 n^{m(p)+m(q)-2}d_p^nd_q^n,$$
    where $m(p)$ and $m(q)$ denote the multiplicity of the spectral radius of the action of $f^*$ on, respectively, $H^{p,p}(X,\C)$ and $H^{q,q}(X,\C)$.

    In particular, the spectral radius $r_{p,q}$ of the action of $f^*$ on $H^{p,q}(X,\C)$ satisfies
    $$r_{p,q}\le \sqrt{d_pd_q}.$$
\end{prop}
\vspace{-2\baselineskip}

\begin{proof}
    We only have to prove the assertion for $q\not=p$. We do only the case $p>q$. The other case can be treated in the same way. Let $\varphi$ be a smooth $(p,q)$-form on $X$. By Poincaré duality, it suffices to give an estimate of $|\langle(f^n)^*(\varphi),\psi\rangle|^2$, where $\psi$ belongs to a fixed finite set (that depends only on $X$) of smooth $(k-p,k-q)$-forms. We can assume $\varphi=\theta\wedge\Omega$ and $\psi=\theta'\wedge\Omega'$, where $\theta$ and $\theta'$ are smooth forms of bidegrees $(p-q,0)$ and $(0,p-q)$ respectively, and $\Omega$ and $\Omega'$ are smooth positive forms of bidegrees $(q,q)$ and $(k-p,k-p)$ respectively. Indeed, $\varphi$ and $\psi$ can be written as finite sums of forms of the considered types. Define $\tilde{\Omega}:=\theta\wedge\bar{\theta}\wedge\Omega$ and $\tilde{\Omega}':=\theta'\wedge\bar{\theta}'\wedge\Omega'$. The form $\tilde{\Omega}$ is of bidegree $(p,p)$ and the form $\tilde{\Omega}'$ is of bidegree $(k-q,k-q)$. By the Cauchy-Schwarz inequality, we have
    \begin{align} \label{cs-applied}
        |\langle(f^n)^*(\varphi),\psi\rangle|^2&=|\lp [\Gamma_n],\pi_2^*(\theta\wedge\Omega)\wedge\pi_1^*(\theta'\wedge\Omega')\rp|^2\nonumber\\
        &\le |\lp [\Gamma_n],\pi_2^*(\tilde\Omega)\wedge\pi_1^*(\Omega')\rp|\cdot|\lp [\Gamma_n],\pi_2^*(\Omega)\wedge\pi_1^*(\tilde\Omega')\rp|\nonumber\\
        &=|\langle(f^n)^*(\tilde\Omega),\Omega'\rangle||\langle(f^n)^*(\Omega),\tilde\Omega'\rangle|.
    \end{align}
    Since $\Omega,\Omega',\tilde\Omega$ and $\tilde\Omega'$ are positive smooth forms, we have
    \begin{align} \label{Omega-omega}
        |\langle(f^n)^*(\tilde\Omega),\Omega'\rangle||\langle(f^n)^*(\Omega),\tilde\Omega'\rangle|&\lesssim|\langle(f^n)^*(\omega^p),\omega^{k-p}\rangle|\cdot|\langle(f^n)^*(\omega^q),\omega^{k-q}\rangle|\nonumber\\
        &=\|(f^n)^*(\omega^p)\|\cdot\|(f^n)^*(\omega^q)\|\lesssim n^{m(p)+m(q)-2}d_p^nd_q^n.
    \end{align}
    Combining \eqref{cs-applied} and \eqref{Omega-omega}, we get
    $$|\langle(f^n)^*(\varphi),\psi\rangle|^2\lesssim n^{m(p)+m(q)-2}d_p^nd_q^n,$$
    where the implicit constant depends on $\psi$ and $\varphi$. The assertion follows.
\end{proof}

\subsection{Non-invertible linear maps}\label{linear-alg}

To end this section, we state some linear algebra results that we will need in the sequel. This is a version of \cite[Section 2.2]{DS05JAMS} in the non-invertible case. Let $L$ be a (not necessarily invertible) linear transformation of a real vector space $V$ of dimension $h$. Let $\cK$ be a closed convex cone with non-empty interior which generates $V$ and satisfies $\cK\cap -\cK=\{0\}$. Suppose $L(\cK)\subset \cK$.

Consider the complexification $V^\C=V\otimes_\R\C$ and the extension of $L$ on $V^\C$. We can find a complex basis such that the matrix of $L$ is given by its Jordan canonical form. More precisely, there exist invariant subspaces $V_i^\C\subseteq V^\C$ such that we have the decomposition
\[
    V^\C=\bigoplus_{i=1}^r V_i^\C,
\]
and the matrix of $L|_{V_i^\C}$ is given by a Jordan block $J_{m_i,\lambda_i}$ of size $m_i$ and eigenvalue $\lambda_i$. We can arrange the order of these blocks such that for all $1\leq i\leq r-1$, we either have $|\lambda_i|>|\lambda_{i+1}|$ or $|\lambda_i|=|\lambda_{i+1}|$ and $m_i\geq m_{i+1}$. Then $\lambda:=|\lambda_1|$ is the {\it spectral radius} of $L$ and is in fact an eigenvalue by Perron-Frobenius theorem. The integer $m:=m_1$ is called the {\it multiplicity of the spectral radius}. We have $\|L^n\|\sim n^{m-1}\lambda^n$. Notice that since $L$ is not necessarily invertible, some Jordan blocks may correspond to the eigenvalue $0$. Such blocks are nilpotent, and in our chosen ordering they are placed in the lower-right corner of the Jordan form.

The eigenspace of $L|_{V_i^\C}$ is given by a complex line denoted by $F_i^\C$. We say a Jordan block $J_{m_i,\lambda_i}$ is {\it dominant} if $(m_i,|\lambda_i|)=(m,\lambda)$. Suppose the dominant Jordan blocks are $J_{m_1,\lambda_1},\dots,J_{m_\nu,\lambda_\nu}$. Define
$$F^\C=\bigoplus_{i=1}^\nu F_i^\C\qquad\text{and}\qquad H^\C=\bigoplus_{1\leq i\leq \nu,\lambda_i=\lambda}F_i^\C.$$
Then the {\it dominant eigenspace} $F$ and {\it strictly dominant eigenspace} $H$ are defined by
$$F=F^\C\cap V\qquad\text{and}\qquad H=H^\C\cap V$$
respectively. They are both invariant by $L$.

For each $1\leq i\leq \nu$, there exists a unique $\theta_i\in\Sb^1$ such that $\lambda_i=\lambda\exp(2\pi\theta_i)$. Define $\theta:=(\theta_1,\dots,\theta_\nu)\in\Sb^\nu$ to be the {\it dominant direction} of $L$. Define the normalized operator $\Lambda_n:=n^{1-m}\lambda^{-n}L^n$ for each $n\geq 1$. The following result can be proven by considering the powers of Jordan blocks of $L$. 

\begin{prop}\label{surj-linear}
    Any subsequence $(\Lambda_{n_i})_{i\in\N}$ converges if and only if $(n_i\theta)$ converges. Take $a\in V$. If $(\Lambda_{n_i}a)$ converges to some $c$, then $c\in F$. Any limit of $(\Lambda_n)$ is a surjective linear map from $V$ to $F$.
\end{prop}

 Let $f$ be a holomorphic correspondence on a compact K\"ahler manifold $(X,\omega)$ of dimension $k$. In the next section, we will apply the results of the present section to the action of $f$ to the cohomology groups of $X$. Namely, we will let $L=f^*$ acting on $V=H^{q,q}(X,\R)$ and $\cK$ be the cone of classes in $H^{q,q}(X,\R)$ which can be represented by a positive closed $(q,q)$-form. Since $f^*$ preserves positive closed currents and its action passes in cohomology, it also preserves $\cK$.

\section{Construction and regularity of Green currents}\label{sec:constr_Green}

Throughout all this section, let $f$ be a holomorphic correspondence on a compact K\"ahler manifold $(X,\omega)$ of dimension $k$. We also fix an integer $1\leq q\leq k$ such that $d_{q-1}(f)<d_q(f)$. In Section \ref{cons_of_green} we construct the Green currents associated with $f$. In Section \ref{reg_of_green} we will discuss the precise regularity of their super-potentials.

\subsection{Construction of Green currents} \label{cons_of_green}
We adapt the proof of \cite[Theorem 4.2.1]{DS10JAG}. In particular, we need to handle the difficulties given by the non-invertibility of $f$.

\begin{thm}\label{thm: green-exist}
    Let $F$ denote the real dominant subspace of $f^*$ on $H^{q,q}(X,\R)$. Then each class $c\in F$ has a representative $T_c\in\cD_q$ with a continuous super-potential which depends linearly on $c$ and satisfies $f^*(T_c)=T_{f^*(c)}$. 
\end{thm}

 To prove the theorem, we will use the following equidistribution result for closed currents with continuous super-potentials, which is the generalization of \cite[Proposition 4.2.2]{DS10JAG} to correspondences. The strategy of the proof is similar, but since in our case $f$ is not invertible, the proof of equality \eqref{eq:SP-converg} below is more delicate and requires some approximation arguments.

\begin{prop} \label{prop: equi_cont-SP}
    Let $S$ be a current in $\cD_q$ with a continuous super-potential. If the sequence $n_i^{1-m}d_q^{-n_i}(f^{n_i})^*\{S\}$ converges to some class $c$, then $n_i^{1-m}d_q^{-n_i}(f^{n_i})^*(S)$ converges SP-uniformly to some $T_c\in \cD_q$ which only depends on $c$.
\end{prop}

\begin{proof}
    Put $S_n:=(f^n)^*S$ which also has a continuous super-potential. Recall that we choose $\alpha=(\alpha_1,\dots,\alpha_h)$ such that $\{\alpha\}=(\{\alpha_1\},\dots,\{\alpha_h\})$ is a basis of $H^{q,q}(X,\R)$. Let $M$ denote the $h\times h$ matrix whose column of index $i$ is given by the coordinates of $f^*\{\alpha_i\}$ with respect to $\{\alpha\}$. Define $\cU=(\cU_{f^*(\alpha_1)},\dots,\cU_{f^*(\alpha_h)})$. Let $A=\,^t{(a_1 , \dots, a_h)}$ denote the coordinates of $\{S\}$ in the basis $\{\alpha\}$. Then $\{S_n\}=\{\alpha\} M^nA$.
    
    First, we prove the equality 
    \begin{equation}\label{eq:SP-converg}
        \cU_{S_n}=\sum_{l=0}^{n-1}(\cU\circ (f^l)_*) M^{n-l-1} A+\cU_S\circ (f^n)_*.
    \end{equation}
    We proceed by induction. The base step $n=0$ is obviously true. Suppose we have \eqref{eq:SP-converg} for $n\ge0$, we need to prove it for $n+1$. Let $R$ be a smooth form in $\cD_{k-q+1}^0$ and $U_R$ be a smooth $\alpha$-normalized potential of $R$. Then by definition
    \begin{equation*}
        \cU_{S_{n+1}}(R)=\lp f^*(S_n), U_R\rp=\lp \pi_2^*(S_n)\wedge [\Gamma], \pi_1^*(U_R)\rp.
    \end{equation*}
    Since $S_n$ has continuous super-potential, the above wedge-product is well-defined and continuous. By \cite[Theorem 2.1]{DNV18AM}, we can find a sequence of smooth closed forms $([\Gamma]_j)$ in $\cD_k(X\times X)$ which converges to $[\Gamma]$ with respect to the $*$-topology. We have 
    \[
        \cU_{S_{n+1}}(R)=\lim_{j\to\infty}\lp \pi_2^*(S_n)\wedge [\Gamma]_j, \pi_1^*(U_R)\rp=\lim_{j\to\infty}\lp S_n,(\pi_2)_*([\Gamma]_j\wedge \pi_1^*(U_R))\rp.
    \]
    Define $U^j=(\pi_2)_*([\Gamma]_j\wedge \pi_1^*(U_R))$. Let $U^j_0=U^j-\sum_{i=1}^h\lp U^j,\alpha_i\rp\check{\alpha}_i$. Then $U^j_0$ is an $\alpha$-normalized potential of $(\pi_2)_*([\Gamma]_j\wedge \pi_1^*(R))$. Since $S_n$ has a continuous super-potential and $(\pi_2)_*([\Gamma]_j\wedge \pi_1^*(R))$ converges to $f_*(R)=(\pi_2)_*([\Gamma]\wedge \pi_1^*(R))$ in the $*$-topology, we have
    \[
        \lim_{j\to\infty}\lp S_n, U_0^j\rp=\lim_{j\to\infty}\cU_{S_n}((\pi_2)_*([\Gamma]_j\wedge \pi_1^*(R)))=\cU_{S_n}(f_*(R)).
    \]
    Therefore, by the inductive hypothesis, we have
    \begin{align*}
        \cU_{S_{n+1}}(R)&=\cU_{S_n}(f_*(R))+\lim_{j\to\infty}\lp S_n, \sum_{i=1}^h\lp U^j,\alpha_i\rp\check{\alpha}_i\rp\\
        &=\sum_{l=1}^{n}(\cU\circ (f^l)_*(R)) M^{n-l} A+\cU_S\circ (f^{n+1})_*(R)+\lim_{j\to\infty}\lp S_n, \sum_{i=1}^h\lp U^j,\alpha_i\rp\check{\alpha}_i\rp.
    \end{align*}
    
    It remains to prove that $\lim_{j\to\infty}\lp S_n, \sum_{i=1}^h\lp U^j,\alpha_i\rp\check{\alpha}_i\rp=\cU(R) M^n A.$ Notice that
    \[
        \lim_{j\to\infty}\lp U^j,\alpha_i\rp=\lim_{j\to\infty}\lp U_R, (\pi_1)_*(\pi_2^*(\alpha_i)\wedge[\Gamma]_j)\rp=\lp f^*(\alpha_i), U_R\rp=\cU_{f^*(\alpha_i)}(R).
    \]
    Therefore, $\lim_{j\to\infty}\lp S_n, \sum_{i=1}^h\lp U^j,\alpha_i\rp\check{\alpha}_i\rp=\cU(R)\lp S_n,\check{\alpha}\rp$ where $\lp S_n,\check{\alpha}\rp$ represents the column vector whose $i$-th entry is $\lp S_n,\check{\alpha}_i\rp$. Then \eqref{eq:SP-converg} for $n+1$ follows from the fact that $S_n$ is cohomologous to $\alpha M^n A$. 

    Let $\cU_{n_i}$ be the super-potential of $n_i^{1-m}d_q^{-n_i} S_{n_i}$. By \eqref{eq:SP-converg}, we have
    \begin{equation} \label{eq:SP-converge-two}
        \cU_{n_i}=\sum_{l=0}^{n_i-1}(\cU\circ (f^l)_*) \frac{M^{n_i-l-1} A}{n_i^{m-1}d_q^{n_i}}+n_i^{1-m}d_q^{-n_i}\cU_S\circ (f^n)_*.
    \end{equation}
    Take $\delta$ with $d_{q-1}<\delta<d_q$. We have $\|(f^n)_*(R)\|_*\lesssim \delta^n\|R\|_*$. Since $\cU_S$ is continuous,  the last term above converges to 0 uniformly on a $*$-bounded set.
    
     We identify $c$ with the column vector associated to $c$ with respect to the basis $\{\alpha\}$. By the definitions of $d_q$ and $m$, we have $n_i^{1-m}d_q^{-n_i}M^{n_i}A\to c$ as $i$ goes to infinity. By Proposition \ref{surj-linear}, $c\in F$. Moreover, we have $n_i^{1-m}d_q^{-n_i}M^{n_i-l-1}A\to (M|_F)^{-l-1}c$ for any $l\geq 0$ as $i$ goes to infinity. This can be proven by considering the Jordan canonical form of $M$.
     
     Observe that $\|(M|_F)^{-l-1}c\|\lesssim d_q^{-l}$.  Therefore, the following limit exists:
     \begin{equation} \label{green_superpot_formula}
        \cU_{T_c}:=\lim_{i\to\infty}\cU_{n_i}=\sum_{l=0}^{+\infty} (\cU\circ (f^l)_*) (M|_F)^{-l-1}c.
    \end{equation}
    The convergence is also uniform on $*$-bounded subsets. We define the super-potential of $T_c$ to be the above limit and the action of $T_c$ on a test smooth $(k-q,k-q)$-form $\varphi$ can be defined by
    \[
        \lp T_c,\varphi\rp:=\sum_{i=1}^h c_i\lp\alpha_i,\varphi\rp+ \cU_{T_c}(\ddc\varphi)
    \]
    when $c=(c_1,\dots,c_h)$. From the above and \eqref{green_superpot_formula}, it follows that $T_c$ depends only on $c$. We also have
    $$\sum_{i=1}^h c_i\lp\alpha_i,\varphi\rp+ \cU_{T_c}(\ddc\varphi)=\lim_{j\rightarrow\infty}\lp\alpha,\varphi\rp\frac{M^{n_j}A}{n_j^{m-1}d_q^{n_j}}+\cU_{n_j}(\ddc\varphi)=\lim_{j\rightarrow\infty}\lp n_j^{1-m}d_q^{-n_j}(f^{n_j})^*(S),\varphi\rp,$$
    where $\lp\alpha,\varphi\rp$ represents the row vector whose $i$-th entry is $\lp\alpha_i,\varphi\rp$, and we used the fact that a closed current is determined by its cohomology class and super-potential. We deduce that the sequence $n_j^{1-m}d_q^{-n_j}(f^{n_j})^*(S)$ converges to $T_c$ in the sense of currents. As a consequence of \eqref{green_superpot_formula}, it also converges SP-uniformly. The proof is complete.
\end{proof}

\begin{proof}[Proof of Theorem \ref{thm: green-exist}]
    Let $(n_i)$ be a subsequence such that $(n_i^{1-m}d_q^{-n_i}(f^{n_i})^*)$ converges on $H^{q,q}(X,\R)$. By Proposition \ref{surj-linear}, for any $c\in F$, we can find a smooth closed $(q,q)$-form $S$ such that $n_i^{1-m}d_q^{-n_i}(f^{n_i})^*\{S\}$ converges to $c$. Therefore, Proposition \ref{prop: equi_cont-SP} implies the existence of $T_c$ which depends only on $c$.
    
    Recall that the push-forward $f_*$ is continuous, and $\cU$ is continuous too, see for instance \cite[Proposition 4.1 and Lemma 4.3]{DNV18AM}. So, each term in the sum in \eqref{green_superpot_formula} is continuous on any $*$-bounded subset of $\cD_{k-q+1}^0$. It follows that $\cU_{T_c}$ is the uniform limit of continuous functions, hence it is continuous.
    
    By Proposition \ref{prop: equi_cont-SP}, $T_c$ depends linearly on $S$ and hence depends linearly on $c$. Finally, if $n_i^{1-m}d_q^{-n_i}(f^{n_i})^*\{S\}$ converges to $c$, then $n_i^{1-m}d_q^{-n_i}(f^{n_i+1})^*\{S\}$ converges to $f^*(c)$. Hence, by Proposition \ref{prop: equi_cont-SP}, we have
    \[
        f^*(T_c)=\lim_{i\to\infty} n_i^{1-m}d_q^{-n_i}(f^{n_i+1})^*(S)=T_{f^*(c)}.
    \]
    This concludes the proof.
\end{proof}

\subsection{Regularity of Green currents} \label{reg_of_green}

In this subsection, we give a more precise estimate of the regularity of the super-potentials of the Green currents $T_c$.  

\smallskip

We will need the following \L{}ojasiewicz-type inequality.

\begin{lm} \label{loja}
    Let $X_1,X_2$ be K\"ahler manifolds with $X_1$ compact, and let $A\subset X_1\times X_2$ be an analytic subset. Suppose that the canonical projection $\pi_1:X_1\times X_2\longrightarrow X_1$, when restricted to $A$, is a ramified covering of maximal multiplicity $\rho$ and with finite fibres of cardinality $d$ (counted with multiplicity).

    There exists a constant $C=C(X_1,X_2,A)$ such that for every $z\in X_1$ we have
    \begin{equation} \label{loja_set}
        \dist\big(\pi_1^{-1}(z)\cap A, x\big) \le C\cdot\dist\big(z,\pi_1(x)\big)^{1/\rho}.
    \end{equation}

    Moreover, for every $y,y'\in X_1$ we can write
    $$\pi_1^{-1}(y)\cap A=\{y_1,\dots,y_d\}\quad\text{and}\quad\pi_1^{-1}(y')\cap A=\{y'_1,\dots,y'_d\},$$
    where the preimages are counted with multiplicity, so that
    \begin{equation} \label{loja_holder}
        \dist(y_j,y_j')\le C\cdot\dist(y,y')^{1/\rho}\quad\text{ for every }1\le j\le d.
    \end{equation}
\end{lm}
\vspace{-2\baselineskip}

\begin{proof}
    Since $X_1$ is compact, it suffices to show that for every $z'\in X_1$ we can find an open neighbourhood $U_{z'}\ni z'$ such that \eqref{loja_set} holds for every $z\in U_{z'}$ and \eqref{loja_holder} holds for every $y,y'\in U_{z'}$. We can then cover $X_1$ with a finite number of such neighbourhoods.

    \smallskip

    Since $\pi_1|_A$ is a ramified covering, for every $z'\in X_1$ and every $w\in \pi_1^{-1}(z')\cap A$ we can find a neighbourhood $V_w=D_w\times D$ of $w$, where $D_w$ and $D$ are small discs contained in a chart such that $z'\in D_w$ and $\pi_1\vert_{V_w\cap A}$ is proper. It is also a ramified covering of degree $\rho_w$. We are then in the condition to apply \cite[Lemma 4.3]{DS08ASENS}, which tells us the following.

    \begin{lm}
        There is an open set $\tilde{D}_w\subset D_w$, containing $z'$, and a constant $c_w$ such that for every $z\in\tilde{D}_w$ and $x\in A$ with $\pi_1(x)\in\tilde{D}_w$ we have
        $$\dist\big(\pi_1\vert_{V_w}^{-1}(z)\cap A, x\big) \le C_w\cdot\dist\big(z,\pi_1(x)\big)^{1/\rho_w}.$$

        Moreover, for every $y,y'\in \tilde{D}_w$ we can write
        $$\pi_1\vert_{V_w}^{-1}(y)\cap A=\{y_1,\dots,y_{\rho_w}\}\quad\text{and}\quad\pi_1\vert_{V_w}^{-1}(y')\cap A=\{y'_1,\dots,y'_{\rho_w}\},$$
        where the preimages are counted with multiplicity, so that
        $$\dist(y_j,y_j')\le C_w\cdot\dist(y,y')^{1/\rho_w}\quad\text{ for every }1\le j\le \rho_w.$$
    \end{lm}
    It then suffices, for every $z'\in X_1$, to choose $U_{z'}=\displaystyle \bigcap_{w\in\pi_1^{-1}(z')\cap A} \tilde{D}_w$. Observe that, since $\pi_1\vert_{V_w\cap A}$ is a ramified covering of degree $\rho_w$ and $\displaystyle \sum_{w\in\pi_1^{-1}(z')\cap A} \rho_w=d$, the set $\displaystyle\bigcup_{w\in\pi_1^{-1}(z')\cap A}\pi_1\vert_{V_w}^{-1}(z)\cap A$ contains exactly $d$ point (counted with multiplicity) for every $z\in U_{z'}$, i.e., it coincides with $\pi_1^{-1}(z)\cap A$. It follows that \eqref{loja_set} holds for every $z\in U_{z'}$ and \eqref{loja_holder} holds for every $y,y'\in U_{z'}$, for constants $\tilde{C}=\displaystyle\max_{w\in\pi_1^{-1}(z')\cap A} \{C_w\}$ and $\tilde{\rho}=\displaystyle\max_{w\in\pi_1^{-1}(z')\cap A} \{\rho_w\}$ (we can always choose the neighbourhoods sufficiently small, so that $\dist(y,y')<1$ for every $y,y'\in U_{z'}$).
    
    By compactness, we just need to consider a finite number of $z'$, and every $z'$ has a finite number of points $w$ in its fibre over $A$. Therefore, for \eqref{loja_set} it suffices to take as $C$ the maximum of the constants $C_w$, and for \eqref{loja_holder} we take a sufficiently large constant to also deal with couples of points that are far apart, i.e., that do not belong to the same $U_{z'}$. Moreover, $\tilde\rho$ is less than or equal to the maximal multiplicity $\rho$ of $\pi_1$ when restricted to $A$. The proof is complete.
\end{proof}

We now consider $X\times X$ with $\pi_1$ the projection map to the first factor and $\Gamma$ the graph of $f$ in $X\times X$.

\begin{defn} \label{local_mult_def}
    The \emph{local multiplicity} $\rho:=\rho(f)$ of $f$ is the maximal multiplicity of $\pi_1|_\Gamma$ seen as a ramified covering. Set $\kappa:=1/\big(25k(4\rho)^q\big)$.
\end{defn}

We have the following version of \cite[Lemma 5.2]{V25Ax} for correspondences.

\begin{prop} \label{super_of_pf}
    Let $\Omega$ be a bounded closed $(q,q)$-form with $\|\Omega\|_\infty\le 1$. Then $\cU_{f^*(\Omega)}$ is $(C_1,\kappa)$-H\"older-continuous, where the constant $C_1>0$ may depend on $f$, but is independent of $\Omega$.
\end{prop}
\vspace{-0.3\baselineskip}

\begin{proof}
    We will prove the assertion in four steps.\\

    \noindent\emph{\textbf{Step 1}: $f^*(\omega)$ has a $1/(2\rho)$-H\"older-continuous super-potential.}

    \medskip
    
    Set $h:=\dim{H^{1,1}(X,\R)}$, and let $\alpha=(\alpha_1,\dots,\alpha_h)$ be as in the beginning of Section \ref{sec:constr_Green} for $q=1$. Take real numbers $a_1,\dots,a_h$ such that $\{f^*(\omega)\}=\displaystyle\sum_{j=1}^h a_j\cdot\{\alpha_j\}$. Then there exists a function $u$ such that $\ddc u=f^*(\omega)-\displaystyle\sum_{j=1}^h a_j\cdot\alpha_j$. We want to show that $u$ is $1/\rho$-H\"older-continuous.

    The assertion is local. So, for every point $z\in X$, we can consider a small open neighbourhood $U\subseteq X$ of $z$ such that:
    \begin{itemize}
        \item there exist smooth functions $u_j$ such that $\ddc u_j=\alpha_j$ in $U$ for every $j$;
        \item the image $f(U)$ is the union of small open sets, and there is a smooth function $\tilde{u}$ such that $\omega=\ddc\tilde{u}$ in $f(U)$.
    \end{itemize}
    Define $v:=f^*(\tilde{u})-\displaystyle\sum_{j=1}^h a_j\cdot u_j$. Then we have $\ddc(v-u)=0$ in $U$, so $v-u$ is harmonic, hence smooth in $U$. To prove that $u$ is $1/\rho$-H\"older-continuous, we just have to prove it for $v$. Since the $u_j$'s are smooth, it suffices to prove that $f^*(\tilde{u})$ is $1/\rho$-H\"older-continuous. From Lemma \ref{loja} we get
    \begin{align*}
        |\big(f^*(\tilde{u})\big)(y)-\big(f^*(\tilde{u})\big)(y')|&=\left|\sum_{j=1}^d \tilde{u}\big(\pi_2(y_j)\big)-\sum_{j=1}^d \tilde{u}\big(\pi_2(y_j')\big)\right|\le \|\tilde{u}\|_{\cC^1}\sum_{j=1}^d \dist\big(\pi_2(y_j),\pi_2(y_j')\big) \\
        &\lesssim \sum_{j=1}^d \dist(y_j,y_j')\lesssim \dist(y,y')^{1/\rho}.
    \end{align*}
    Thus, $f^*(\tilde{u})$ is $1/\rho$-H\"older-continuous, and so is $u$.
    
    Consider now a smooth current $R\in \cD_k^0$ with normalized potential $U_R$. We have
    $$\cU_{f^*(\omega)}(R)=\langle f^*(\omega), U_R\rangle=\langle f^*(\omega)-\sum_{j=1}^h a_j\cdot\alpha_j, U_R\rangle=\langle \ddc u, U_R\rangle=\langle u, R\rangle,$$
    where in the second equality we used that $U_R$ is normalized. Since $u$ is $1/\rho$-H\"older-continuous, using Proposition \ref{interpol} we get
    $$|\cU_{f^*(\omega)}(R)|=|\langle u, R\rangle|\lesssim \|R\|_{-1/\rho} \lesssim \|R\|_{-2}^{1/(2\rho)}.$$
    Since we can extend the estimate to all of $\cD_k^0$ by continuity, we get that $\cU_{f^*(\omega)}$ is $1/(2\rho)$-H\"older-continuous, as desired.\\

    \noindent\emph{\textbf{Step 2}: $\big(f^*(\omega)\big)^q$ has a $\frac{1}{2^{q-1}(2\rho)^q}$-H\"older-continuous super-potential.}

    \medskip
    
    Step 2 follows from Step 1 by induction applying Lemma \ref{DS10JAG-3_4_2} to the current $f^*(\omega)$ and its powers.\\

    \noindent\emph{\textbf{Step 3}: there exists a constant $C>0$ such that, for every positive closed $(q,q)$-current $\Omega'$ with $\Omega'\le \omega^q$, we have that $f^*(\Omega')$ has a $\big(C,\frac{1}{25k(4\rho)^q}\big)$-H\"older-continuous super-potential.}

    \medskip
    
    Indeed, from Lemma \ref{wpb} we have $f^*(\Omega')\le f^*(\omega^q)\le  \big(f^*(\omega)\big)^q$. Step 3 then follows from Step 2 by applying Lemma \ref{DNV18AM-1_1} with $S'=f^*(\Omega')$ and $S=\big(f^*(\omega)\big)^q$.\\

    \noindent\emph{\textbf{Step 4}: conclusion.}

    \medskip
    
    Since $\|\Omega\|_\infty\le1$, we can write $\Omega=\Omega^+-\Omega^-$ with $\Omega^\pm$ positive closed currents in the same cohomology class and $\Omega^\pm\lesssim\omega^q$. It follows that
    $$f^*(\Omega)=f^*(\Omega^+)-f^*(\Omega^-).$$
    By the definition of super-potentials, we have
    $$\mathcal{U}_{f^*(\Omega)}=\mathcal{U}_{f^*(\Omega^+)}-\mathcal{U}_{f^*(\Omega^-)}.$$
    We conclude the proof applying Step 3 with $\Omega^\pm$ instead of $\Omega'$.
\end{proof}

\begin{cor} \label{pf_is_hold}
    Take $R$ in $\cD_{k-q+1}^0$, $0<q\le k$. Then there is a constant $C_2>0$, independent of $R$, such that
    \[
        \|f_*(R)\|_{-2} \le C_2\|R\|_{-2}^\kappa,
    \]
    where $\kappa$ is as in Definition \ref{local_mult_def}.
\end{cor}

\begin{proof}
    Given a smooth $(q-1,q-1)$-form $\Omega$ with $\|\Omega\|_{\cC^2}\le 1$, we have to estimate $|\langle f_*(R),\Omega\rangle|$ independently of $\Omega$. Given a normalized potential $U_R$ of $R$, we have
    $$\langle f_*(R),\Omega\rangle=\langle R,f^*(\Omega)\rangle=\langle U_R,\ddc f^*(\Omega)\rangle=\langle U_R,f^*(\ddc \Omega)\rangle=\cU_{f^*(\ddc\Omega)}(R),$$
    where $\ddc\Omega$ is a smooth $(q,q)$-form with $\|\ddc\Omega\|_\infty\lesssim \|\Omega\|_{\cC^2}\le 1$. To conclude the proof, it then suffices to apply Proposition \ref{super_of_pf}.
\end{proof}

The following result gives part (2) of Theorem \ref{main-thm-1}. Recall that $F$ is the dominant eigenspace.

\begin{thm} \label{who-is-r}
    For every class $c\in F$, the super-potential $\cU_{T_c}$ of the Green current $T_c$ is $\log^r$-continuous for some $r>0$.
\end{thm}

We will follow the general strategy of the proof of \cite[Lemma 3.1]{DW23AMP}.

\begin{proof}
    Recall that, by Definition \ref{defn: regularity-SP}, we need to prove that
    \begin{equation} \label{log_of_green}
        |\cU_{T_c}(R)|\le\frac{L}{(1+|\log{\|R\|_{-2}}|)^r}
    \end{equation}
    for some $L,r>0$ and for all $R\in\cD^0_{k-q+1}$ with $\|R\|_*\le1$.
    
    Using the same notation of the proof of Proposition \ref{prop: equi_cont-SP}, we have
    $$\cU_{T_c}(R)=\sum_{l=0}^{+\infty} (\cU\circ (f^l)_*(R)) (M|_F)^{-l-1}c.$$
    As before, we identify $c$ with the column vector associated to $c$ with respect to the basis $\{\alpha\}$. Observe that we have $\|(M|_F)^{-l-1}c\|\lesssim d_q^{-l}$. Therefore, we get
    \begin{align*}
        |\cU_{T_c}(R)|=\left|\sum_{l=0}^{+\infty} (\cU\circ (f^l)_*(R)) (M|_F)^{-l-1}c\right|&\lesssim \sum_{l=0}^{+\infty}d_q^{-l}\sum_{j=1}^h |\cU_{f^*(\alpha_j)}\circ (f^l)_*(R)|\\
        &=\sum_{j=1}^h\sum_{l=0}^{+\infty}d_q^{-l} |\cU_{f^*(\alpha_j)}\circ (f^l)_*(R)|.
    \end{align*}
    To obtain the desired estimate \eqref{log_of_green}, we just need to give an estimate of the infinite sum for every $j=1,\dots,h$. From now on, we fix $j$. From Proposition \ref{super_of_pf} we have that $\cU_{f^*(\alpha_j)}$ is $(C_1,\kappa)$-H\"older-continuous for some $C_1>0$. From Corollary \ref{pf_is_hold} we also have that $f_*$ is $(C_2,\kappa)$-H\"older-continuous with respect to $\|\cdot\|_{-2}$ for some $C_2>0$.

    Denote $\xi:=\|R\|_{-2}$. Since $\cU_{T_c}$ is continuous, we only need to consider $\xi<1/4$. Otherwise, the desired inequality holds for $L$ large enough. We have $1+|\log{\xi}|\le2|\log\xi|$. So it is enough to prove
    \begin{equation}\label{dw-est}
        \sum_{l=0}^{+\infty}d_q^{-l} |\cU_{f^*(\alpha_j)}\circ (f^l)_*(R)| \le \frac{\tilde{L}}{|\log{\xi|^r}}\qquad\text{for some }\tilde{L},r>0.
    \end{equation}
    Fix a constant $C_3>1$ larger than $C_2$. By applying Corollary \ref{pf_is_hold} inductively, for every $n\ge1$ we have
    $$\|(f^l)_*(R)\|_{-2}\le C_2^{1+\kappa+\dots+\kappa^{l-1}}\xi^{\kappa^l}\le C_3^{\frac{1}{1-\kappa}}\xi^{\kappa^l}.$$
    Recall that we have $\|(f^l)_*(R)\|_*\lesssim \delta^l$ for some $d_{q-1}<\delta<d_q$. Hence, we deduce that we have $|\cU_{f^*(\alpha_j)}\circ (f^l)_*(R)|\lesssim \delta^{l(1-\kappa)}C_3^{\frac{\kappa}{1-\kappa}}\xi^{\kappa^{l+1}}\le \delta^lC_3^{\frac{\kappa}{1-\kappa}}\xi^{\kappa^{l+1}}$ for every $l\ge0$. Set $D:=d_q/\delta$.
    
    After multiplying $\cU_{f^*(\alpha_j)}$ by a constant, one can assume $\displaystyle \sup_{\|R\|_*\le 1}|\cU_{f^*(\alpha_j)}(R)|\le1$ and $C_1\le1$. Therefore, using that $\xi^{\kappa^l}\le\xi^{\kappa^{l'}}$ for $l\le l'$, we get

    \begin{align} \label{dw-1}
        \sum_{l=0}^{+\infty}d_q^{-l} |\cU_{f^*(\alpha_j)}\circ (f^l)_*(R)| &\lesssim\sum_{l=0}^{N-1}d_q^{-l} |\cU_{f^*(\alpha_j)}\circ (f^l)_*(R)|+\sum_{l=N}^{+\infty}D^{-l}\nonumber\\
        &\le\sum_{l=0}^{N-1} D^{-l}C_3^{\frac{\kappa}{1-\kappa}}\xi^{\kappa^{l+1}}+\frac{D^{-N+1}}{D-1}\le C_4(\xi^{\kappa^N}+D^{-N}),
    \end{align}
    where $N$ is any positive integer and $C_4>0$ is a constant depending only on $C_3$, $\kappa$, and $D$.
    
    Now we take
    $$N:=\lfloor(\log|\log{\xi}|)/|2\log{\kappa}|\rfloor\qquad\text{and}\qquad r:=\log{D}/|2\log{\kappa}|.$$

    Since $\kappa<1$ and $\xi<1/4$, we get
    \begin{align}\label{dw-2}
        \xi^{\kappa^N}&\le\xi^{\kappa^{(\log|\log{\xi}|)/|2\log{\kappa}|}}=\exp\big(\kappa^{(\log|\log{\xi}|)/|2\log{\kappa}|}\log\xi\big)=\exp(|\log\xi|^{-1/2}\log\xi)\nonumber\\
        &=\exp(-|\log\xi|^{1/2})\le C_5\exp(-r\log|\log\xi|)=C_5|\log\xi|^{-r},
    \end{align}
    where $C_5$ is a constant depending on $\kappa$, $d_q$ and $\delta$. We also have
    \begin{align}\label{dw-3}
        D^{-N}&\le D^{-(\log|\log{\xi}|)/|2\log{\kappa}|+1}=D\exp\Big(\big(-(\log|\log{\xi}|)/|2\log{\kappa}|\big)\cdot\log D\Big)\nonumber\\
        &=D\exp(-r\log|\log\xi|)=D|\log{\xi}|^{-r}.
    \end{align}
    Combining \eqref{dw-1}, \eqref{dw-2}, and \eqref{dw-3}, we conclude that
    $$\sum_{l=0}^{+\infty}D^{-l} |\cU_{f^*(\alpha_j)}\circ (f^l)_*(R)| \lesssim C_4(C_5+D)|\log{\xi}|^{-r}.$$
    This gives \eqref{dw-est} and completes the proof of the theorem.
\end{proof}

\vspace*{-0.3\baselineskip}

\begin{rmk}
    Denote by $l_q:=d_q/d_{q-1}>1$. By our choice of $D$ and $r$, we can choose $r$ to be any constant strictly smaller than $\log l_q/|2\log\kappa|$. Hence, it depends only on $l_q$ and the multiplicity $\rho$.
\end{rmk}

\section{Adjoint multiplicity and equidistribution results}\label{sec:multi-equidis}

In this section, we consider a holomorphic correspondence $f$ on $X$ with graph $\Gamma$ such that $f^{-1}$ is also holomorphic. This last condition is equivalent to saying that the set $\{x\in X:\dim\pi_2^{-1}(x)\cap\Gamma>0\}$ is empty. We will prove that any positive closed $(q,q)$-current converges to a multiple of the Green current of order $q$ under the action of a holomorphic correspondences $f$ with simple action on cohomology. 

\begin{lm}\label{dual}
    Let $f$ be a holomorphic correspondence on $X$ such that $f^{-1}$ is also holomorphic. Let $T$ be a current in $\cD_{q}^0$ and $R$ be a smooth current in $\cD_{k-q+1}^0$. Then $\cU_T(f_*(R))=\cU_{f^*(T)}(R)$.
\end{lm}
\vspace*{-\baselineskip}

\begin{proof}
    Since $R$ is smooth, $f_*(R)$ has a continuous super-potential. Therefore, $\cU_T(f_*(R))$ is well-defined and we have 
    \[
        \cU_T(f_*(R))=\cU_{f_*(R)}(T)=\cU_R(f^*(T))=\cU_{f^*(T)}(R).
    \]
    Here the second equality comes from \cite[Proposition 4.4]{DNV18AM} and the last equality is \cite[Lemma 2.7]{V25Ax}.
\end{proof}

Now, we give the following definition.

\begin{defn} \label{def:inv_small_mult}
    The {\it adjoint multiplicity} of $f$ denoted by $\delta(f)$ is the maximal local multiplicity of $\pi_2|_\Gamma$. For $0<q\leq k$, we say that $f$ has {\it $q$-small adjoint multiplicity} if $d_q(f)>d_{q-1}(f)$ and $\tilde{\kappa}(f)^{-1}:=25k(4\delta(f))^{k-q+1}<d_q(f)/d_{q-1}(f)$.
\end{defn}

Since the graph of $f^{-1}$ is obtained from $\Gamma$ by involution $(x,y)\mapsto (y,x)$, we have $\delta(f)=\rho(f^{-1})$. 

\begin{lm}\label{n-holder}
    Let $f$ be a holomorphic correspondence on $X$ such that $f^{-1}$ is also holomorphic. Let $R$ be a smooth $(k-q,k-q)$-form with $\|R\|_{\cC^2}\leq 1$. Then $\cU_{(f^n)_*(\ddc R)}$ is $(C_6,\tilde\kappa^n)$-H\"older continuous for some constant $C_6>0$ independent of $R$ and $n$.
\end{lm}

\begin{proof}
    Take a current $\Omega\in\cD_q^0$ with $\|\Omega\|_*\leq 1$. Since $R$ is smooth, $(f^n)_*(\ddc R)$ has a continuous super-potential. So, by Lemma \ref{dual} applied to $f^n$ instead of $f$, we have 
    \[
        \cU_{(f^n)_*(\ddc R)}(\Omega)=\cU_\Omega((f^n)_*(\ddc R))=\cU_{(f^n)^*(\Omega)}(\ddc R)=\lp (f^n)^*(\Omega), R\rp.
    \]
    Applying inductively Corollary \ref{pf_is_hold} to $f^{-1}$, we obtain $\|(f^n)^*(\Omega)\|_{-2}\leq C_6\|\Omega\|_{-2}^{\tilde\kappa^n}$ for some $C_6>0$ independent of $\Omega$. Hence $|\cU_{(f^n)_*(\ddc R)}(\Omega)|\leq C_6\|\Omega\|_{-2}^{\tilde\kappa^n}$. This completes the proof.
\end{proof}

\begin{prop}\label{equidis-exact}
    Let $f$ be a holomorphic correspondence on $X$ with $q$-small adjoint multiplicity. Let $S$ be a current in $\cD_q^0$ with $\|S\|_*\leq 1$. Let $R$ be a smooth $(k-q,k-q)$-form with $\|R\|_{\cC^2}\leq 1$. Then 
    \[
         |\lp d_q^{-n}(f^n)^*(S),R\rp|\leq C_7\lambda_0^n
    \]
    where $C_7>0$ is a constant depending only on $f$, and $0<\lambda_0<1$ can be any constant strictly larger than $d_{q-1}(d_q\tilde{\kappa})^{-1}$.
\end{prop}

\begin{proof}
     Since $f$ has $q$-small adjoint multiplicity, we can choose constants $d_{q-1}<\lambda'<\lambda''<d_q$ such that $1/\tilde{\kappa}<d_q/\lambda''<d_q/\lambda'<d_q/d_{q-1}$. By Lemma \ref{dual}, we have
    \[
        \lp d_q^{-n}(f^n)^*(S),R\rp=d_q^{-n}\cU_{(f^n)^*(S)}(\ddc R)=d_q^{-n}\cU_S((f^n)_*(\ddc R)).
    \]
     When $n$ is large enough, we have $\|(f^n)_*(\ddc R)\|_*\lesssim (\lambda')^n$. By Lemma \ref{n-holder}, $\cU_{(\lambda')^{-n}(f^n)_*(\ddc R)}$ is $(C_6(\lambda')^{-n},\tilde\kappa^n)$-H\"older continuous. By Proposition \ref{skoda}, we have
     \[
          |\lp d_q^{-n}(f^n)^*(S),R\rp|\lesssim Ad_q^{-n}(\lambda')^n(1+\tilde\kappa^{-n}\log^+(C_6(\lambda')^{-n}))\lesssim (\lambda''\tilde\kappa^{-1}d_q^{-1})^n.
     \]
     The proof is complete.
\end{proof}

\begin{cor} \label{green-uniqueness}
    Let $f$ be as in Proposition \ref{equidis-exact}. Then the Green currents are the only positive closed $(q,q)$-currents which are invariant by $d_q^{-1}f^*$.
\end{cor}

\begin{proof}
    Notice that if $T$ is invariant by $d_q^{-1}f^*$, we have $c=\{T\}\in H$, where we recall that $H$ is the strictly dominant eigenspace. Applying Proposition \ref{equidis-exact} to $S=T-T_c$ implies the result.
\end{proof}

We have the following convergence towards the Green currents for all currents in $\cD_q$. Compared with Proposition \ref{prop: equi_cont-SP}, the convergence is no longer SP-uniform. This is due to the fact that $\|R\|_{\cC^2}$ is not uniformly bounded for all smooth $(k-q,k-q)$-form $R$ with $\|\ddc R\|_*\leq 1$.  

\begin{cor}\label{equi-all}
    Let $f$ be as in Proposition \ref{equidis-exact} and $S$ be a current in $\cD_q$. If the sequence $n_i^{1-m}d_q^{-n_i}(f^{n_i})^*\{S\}$ converges to some class $c$, then $n_i^{1-m}d_q^{-n_i}(f^{n_i})^*(S)$ converges to $T_c$. 
\end{cor}

\begin{proof}
    We can write $S=\alpha+S'$, where $S'\in\cD_q^0$ and $\alpha$ is a smooth $(q,q)$-current in $\cD_q$ in the same cohomology class as $S$. Then by Proposition \ref{prop: equi_cont-SP}, $n_i^{1-m}d_q^{-n_i}(f^{n_i})^*(\alpha)$ converges SP-uniformly to $T_c$. The result follows from Proposition \ref{equidis-exact} applied with $S'$ instead of $S$.
\end{proof}

In the following, we assume $f$ is a holomorphic correspondence on $X$ with simple action on cohomology and $f^{-1}$ is also holomorphic. Suppose $d:=d_q(f)$ is the main dynamical degree of $f$. Since we have $(f^{-1})^*=f_*$ and by duality, we get $d_s(f^{-1})=d_{k-s}(f)$ for all $0\leq s\leq k$. Therefore, $f^{-1}$ also has simple action on cohomology with the same main dynamical degree. The limits $T^+=\lim_{n\to\infty} d^{-n}(f^n)^*(\omega^q)$ and $T^-=\lim_{n\to\infty} d^{-n}(f^n)_*(\omega^{k-q})$ exist in the sense of currents. We call them the \textit{main dynamical Green current}. They satisfy the relations $f^*(T^+)=dT^+$ and $f_*(T^-)=dT^-$. We denote by $c^+:=\{T^+\}\in H^{q,q}(X,\R)$ and $c^-:=\{T^-\}\in H^{k-q,k-q}(X,\R)$. By \cite[Lemma 5.4]{DNV18AM}, the intersection number $c^+\smile c^-$ is strictly positive. Moreover, multiplying $\omega$ by a constant we can assume 
    \[
        \lp T^+,T^-\rp=c^+\smile\{\omega^{k-q}\}=\{\omega^q\}\smile c^-=1.
    \]
We also need the following result. See \cite[Lemma 5.2]{DNV18AM}.

\begin{lm}
    Let $L^+$ denote the line spanned by $c^+$. Then exists a hyperplane $K^+$ of $H^{q,q}(X,\R)$, invariant by $f^*$, such that $H^{q,q}(X,\R)=L^+\oplus K^+$ and the spectral radius of $d^{-1}f^*$ on $K^+$, denoted by $r^+$, is strictly less than 1.
\end{lm}

In this case, $L^+$ is exactly the (strictly) dominant eigenspace defined in the previous section. For any $c\in K^+$, we have 
\[
    c\smile c^-=c\smile d^{-n}(f^n)_*(c^-)=d^{-n}(f^n)^*(c)\smile c^-
\]
for any $n\in\N$. By letting $n$ goes to infinity, we deduce that $c\smile c^-=0$.

 Note that the correspondences in the family considered in Corollary \ref{poly-small} all have simple action on cohomology. The following gives Part (1) of Theorem \ref{main-thm-2}.

\begin{cor}\label{equi-exponential}
    Let $f$ be a holomorphic correspondence on $X$ with $q$-small adjoint multiplicity and simple action on cohomology. Let $S$ be a current in $\cD_q$ with $\lp T^-,S\rp=1$. Then for any smooth test $(k-q,k-q)$-form $R$ with $\|R\|_{\cC^2}\leq 1$, we have 
    \[
         |\lp d^{-n}(f^n)^*(S)-T^+,R\rp|\leq C_8\lambda_1^n
    \]
    for some constants $C_8>0$ depending only on $S$ and $f$, and $\lambda_1=\max\{\lambda_0,r_+\}$. 
\end{cor}

\begin{proof}
    Since $\{S\}\smile c^-=\lp T^-,S\rp=1$, we can write $\{S\}=c^++c_0$ for some $c_0\in K^+$ by previous discussion. Therefore, we can find a smooth closed $(q,q)$-form $\beta$ and $S''\in\cD_q^0$ such that $\{\beta\}\in K^+$ and $S=T^++\beta+S'';$. It suffices to prove $d^{-n}(f^n)^*(\beta+S'')$ converges to 0 exponentially fast. The result follows from $\|d^{-n}(f^n)^*(\beta)\|\lesssim r_+^n$ and applying Proposition \ref{equidis-exact} to $S''$.
\end{proof}

\begin{rmk}
    Using interpolation theory \cite{T78Book}, the previous result can be extended to H\"older-continuous functions. Moreover, Proposition \ref{equidis-exact} and Corollary \ref{equi-exponential} hold more generally when the current $R$ has $\log$-H\"older-continuous super-potential and the form $\varphi$ has $\log$-H\"older-continuous coefficients. The proof would be more technical and would essentially follow the same steps of the proofs in \cite{V25Ax}, so we omit it and we focus here on the simpler case.
\end{rmk}

Consider now the holomorphic correspondence $F:=(f,f^{-1})$ on $X\times X$. Its graph $\Gamma_F$ in $X^4$ is the image of $\Gamma_f\times\Gamma_{f^{-1}}$ under the map $\tau:(x,y,z,w)\to(x,z,y,w)$. Denote by $\Delta$ the diagonal of $X\times X$ and ${\Gamma}_{f^n}$ the graph of $f^n$ for every $n\geq 1$. 

\begin{lm}\label{F}
    The following properties hold:
    \begin{enumerate}
        \item $F$ has simple action on cohomology with $d_k(F)=d^2$ being the main dynamical degree;
        \item When $n$ is even, we have $[{\Gamma}_{f^n}]=(F^{n/2})^*[\Delta]$, and when $n$ is odd, we have $[{\Gamma}_{f^n}]=(F^{(n-1)/2})^*[{\Gamma}_{f^1}]$;
        \item $\delta(F)=\delta(f)\cdot\delta(f^{-1})$. 
    \end{enumerate}
\end{lm}
\vspace{-0.3\baselineskip}

\begin{proof}
    Notice that (3) follows directly from the definitions of $\Gamma_F$, $\delta(f)$, $\delta(f^{-1})$ and $\delta(F)$. By K\"unneth formula, for any $0\leq l\leq 2k$, we have
    \[
        H^{l,l}(X\times X,\C)\simeq \bigoplus_{\max\{0,l-k\}\leq r,s\leq \min\{l,k\}} H^{r,s}(X,\C)\otimes H^{l-r,l-s}(X,\C).
    \]
    Therefore, by duality and Proposition \ref{D05JGA-5_7-corr}, we have
    \[
        d_l(F)\leq \max_{\max\{0,l-k\}\leq r,s\leq \min\{l,k\}}\{\sqrt{d_rd_sd_{k-l+r}d_{k-l+s}}\}.
    \]
    We deduce that when $l\neq k$, $d_l(F)<d^2$. Since $F^*(c^+\otimes c^-)=d^2 c^+\otimes c^-$, we also have $d_k(F)=d^2$. Therefore, $d_k(F)$ is strictly larger than all the other dynamical degrees. Again by Lemma \ref{D05JGA-5_7-corr} and since $f$ is simple, the norm of $F^*$ restricted to the subspace $\bigoplus_{\substack{0\leq r,s\leq k,\\ (r,s)\neq (q,q)}}H^{r,s}(X,\C)\otimes H^{k-r,k-s}(X,\C)$ of $H^{k,k}(X\times X,\C)$ is strictly smaller than $d^2$. Finally, the action of $F^*$ on $H^{q,q}(X,\C)\otimes H^{k-q,k-q}(X,\C)$ is given by $f^*\otimes f_*$, which has only one simple eigenvalue $d^2$ of maximal modulus. Therefore, the action of $F$ is simple.

    It remains to prove (2). For simplicity, we only prove the case when $n$ is even. The other case can be treated analogously. Denote by $\pi_1$ and $\pi_2$ the projection maps from $X\times X$ to its factors. Similarly, we define $\Pi_1$ and $\Pi_2$ from $(X\times X)^2$ to $X\times X$. Let $U$ be an open subset in $X$ such that both $U$ and $f^{n/2}(U)$ do not contain critical points of $f^{n/2}$. Then it is enough to prove the desired result on $\pi_1^{-1}(U)$. Notice that 
    \[
        (F^{n/2})^*[\Delta]|_{\pi_1^{-1}(U)}=[\Pi_1(\Gamma_{F^{n/2}}\cap \Pi_2^{-1}(\Delta)\cap (U\times X\times X^2))]=[\Pi_1\circ\tau ((\Gamma_{f^{n/2}}\times\Gamma_{f^{-n/2}})\cap U')]
    \]
    where $U'=\{(x,z,y,z): x\in U\}$. Denote the branches of $f^{n/2}$ on $U$ by $\gamma_1,\dots,\gamma_{d_0^{n/2}}$. Hence $\Gamma_{f^{n/2}}|_{\pi_1^{-1}(U)}=\bigcup_{i=1}^{d_0^{n/2}}\{(x,\gamma_i(x)):x\in U\}$. For each $i$, denote the branches of $f^{n/2}$ on $\gamma_i(U)$ by $\gamma_{i1},\dots,\gamma_{id_0^{n/2}}$. Then we have
    \begin{align*}
        \Pi_1\circ\tau((\Gamma_{f^{n/2}}\times\Gamma_{f^{-n/2}})\cap U')
        &=\bigcup_{j=1}^{d_0^{n/2}}\bigcup_{i=1}^{d_0^{n/2}} \Pi_1\circ\tau(\{(x,\gamma_i(x),\gamma_{ij}(x),\gamma_i(x)):x\in U\})\\
        &=\bigcup_{j=1}^{d_0^{n/2}}\bigcup_{i=1}^{d_0^{n/2}} \{(x,\gamma_{ij}(x)):x\in U\}
    \end{align*}
    which is exactly $\Gamma_{f^n}$ restricted on $\pi_1^{-1}(U)$. This completes the proof.
\end{proof}

\begin{defn}
    Let $f$ be a holomorphic correspondence on $X$ such that $f^{-1}$ is also a holomorphic correspondence. We say that $f$ has \emph{small multiplicity} if the holomorphic correspondence $F=(f,f^{-1})$ has $k$-small adjoint multiplicity.
\end{defn}

As a direct application of Corollary \ref{equi-exponential} and Lemma \ref{F}, we get part (2) of Theorem \ref{main-thm-2}, which is a quantitative version of \cite[Proposition 5.10]{DNV18AM}.  

\begin{prop}
    Let $f$ be a holomorphic correspondence on $X$ with simple action on cohomology. If $f$ has small multiplicity, then the sequence of positive closed $(k,k)$-currents $d^{-n}[\Gamma_{f^n}]$ converges to $T^+\otimes T^-$ exponentially fast.
\end{prop}

\section{Examples} \label{sec-examples}
In this section, we prove that the class of holomorphic correspondences with small adjoint multiplicity is very rich and contains many interesting examples. 

\subsection{Families of holomorphic correspondences}

In order to study our examples, we need to work in the setting of a family of holomorphic correspondences. Before that, we prove a general result about the slicing of an analytic subset whose fibres have constant dimension. Although the proof only uses classical tools, to the best of our knowledge, this result is not explicitly stated in the existing literature. Hence, we will give a complete proof.

\begin{lm} \label{cont_of_analyticsubset}
   Let $X$ be a compact K\"ahler manifold of dimension $n$ and $Y$ be a complex manifold of dimension $d$. Let ${\Gamma}$ be an analytic set of pure dimension $k\geq d$ in $X\times Y$. Suppose for every $y\in Y$, the fibre $\Gamma_y:=\pi_Y^{-1}(y)\cap{\Gamma}$ is an analytic set (regarded in $X$) of pure dimension $k-d$. Then the current $[\Gamma_y]$ changes continuously in $y$, i.e., for a sequence of points $(y_i)$ converging to $y\in Y$, we have $\lim_{i\to\infty}[\Gamma_{y_i}]=[\Gamma_y]$. Moreover, we have $\lp[{\Gamma}],\pi_Y,y\rp=[\Gamma_y]$ for every $y\in Y$.
\end{lm}
    
\begin{proof}
    When $k=d$, the result follows directly from Lemma \ref{loja} applied with $X_1=Y$, $X_2=X$, and $A=\Gamma$. Since the proof of the Lemma uses a local argument and we only need to consider what happens in a neighbourhood of $y$, we do not need $Y$ to be compact. 
    
    Henceforth, we assume $k>d$. Since this is a local property, we can take $X$ and $Y$ to be open subsets in $\C^n$ and $\C^d$ respectively and consider a point $a=0\in X\times Y\subset\C^{n+d}$. Since $\Gamma_0$ is an analytic set of pure dimension $k-d$, by \cite[Lemma 3.2.3]{King71Acta}, we can find coordinates $(x_1,\dots,x_n)$ in $X$ such that for any $(k-d)$-tuple $J=(j_1,\dots,j_{k-d})$ with $1\leq j_1<j_2<\cdots< j_{k-d}\leq n$, the projection from $X$ to $\C^{k-d}$ given by $x\mapsto (x_{j_1},\dots,x_{j_{k-d}})$ is finite at 0 when restricted on $\Gamma_0$. Let $(y_1,\dots,y_d)$ be the coordinate of $Y$. If we take $(x,y):=(x_1,\dots,x_n,y_1,\dots,y_d)$ as a local coordinate of $X\times Y$, the projection $\pi_J:(x,y)\mapsto (x_J,y):=(x_{j_1},\dots,x_{j_{k-d}},y_1,\dots,y_d)$ is finite at $0$ when restricted on ${\Gamma}$. By the Projection Lemma in \cite[Chapter 3, \S 1]{GR1984}, this implies there exists a poly-disc $\D_{X\times Y}\subset X\times Y$ such that $\pi_J|_{{\Gamma}\cap\D_{X\times Y}}$ is a closed map and has finite fibre at every point of $\D_J:=\pi_J(\D_{X\times Y})$. Thus it is proper and is a ramified covering.
    
    Let $\D_{k-d}$ and $\D_Y$ be the poly-disks centered at $0$ in, respectively, $\C^{k-d}$ and $Y$ such that $\D_J=\D_{k-d}\times\D_Y$. For $(x_J,y)\in\D_J$, denote by $m_{(x_J,y)}$ the measure defined by the sum of the Dirac measures given by the points in the fibre $\pi_J|_{{\Gamma}\cap \D_{X\times Y}}^{-1}(x_J,y)$, counted with multiplicity. By Lemma \ref{loja}, for fixed $x_J$, when $y_i$ converges to $y$ in $\D_Y$, we have $m_{(x_J,y_i)}\to m_{(x_J,y)}$. Now for any $y\in \D_Y$, the restricted projection map $\pi|_{\Gamma_y\cap \D_{X\times Y}}:\Gamma_y\cap \D_{X\times Y}\to \D_{k-d}$ is still proper. Therefore, it is also a ramified covering. For $x_J\in\D_{k-d}$ generic, $m_{(x_J,y_i)}$ equals to the slicing of $[\Gamma_y]$ with respect to $\pi|_{\Gamma_y\cap \D_{X\times Y}}$ at the point $x_J$. From the case $k=d$ and the continuity of the slice with respect to the parameter \cite[Theorem 3.3.2]{King71Acta}, the same is true for every $x_J\in\D_{k-d}$. Let $\varphi$ be a smooth test function compactly supported in $\D_{X\times Y}$. Denote by $\omega_J$ the volume form on $\D_{k-d}$. Then, by the bounded convergence theorem, we have
    \begin{align*}
         \lim_{i\to\infty}\lp [\Gamma_{y_i}]\wedge\omega_J,\varphi\rp
         &=\lim_{i\to\infty}\int_{\D_{k-d}}\lp m_{(x_J,y_i)},\varphi\rp \omega_J(x_J)\\
         &=\int_{\D_{k-d}}\lp m_{(x_J,y)},\varphi\rp \omega_J(x_J)=\lp [\Gamma_y]\wedge\omega_J,\varphi\rp.
    \end{align*}
    Since the above is true for any $J$ and $\varphi$, we get the first statement of the lemma.

    Finally, again by \cite[Theorem 3.3.2]{King71Acta}, the slice $\lp[{\Gamma}],\pi_Y,y\rp$ exists for every $y$ and is continuous with respect to $y$. Since $\lp[\Gamma],\pi_Y,y\rp=[\Gamma_y]$ when $y$ is a regular value, we deduce that $\lp[\Gamma],\pi_Y,y\rp=[\Gamma_y]$ for every $y$. 
\end{proof}

\begin{defn}
    Let $\Sigma$ be a connected complex manifold of dimension $k'$. We say that $\{f_t\}_{t\in\Sigma}$ is a \emph{family of holomorphic correspondences} on $X$ if:
    \begin{nlist}
        \item $f_t:X\rightarrow X$ is a holomorphic correspondence on $X$ for every $t\in\Sigma$, with graph $\Gamma_t$;
        \item there is an effective analytic cycle $\Gamma$ of pure dimension $k+k'$ on $\Sigma\times X\times X$ such that $\lp[\Gamma], \pi_\Sigma,t\rp=[\Gamma_t]$ for every $t$. It is called the \emph{graph of the family}.
    \end{nlist}
\end{defn}

Observe that, by Lemma \ref{cont_of_analyticsubset}, we have $\Gamma_t=\Gamma\cap\pi_\Sigma^{-1}(t)$ for every $t\in\Sigma$.

\begin{rmk}\label{the_defs_are_equiv}
Since the composition is defined locally, $\{f^n_t\}_{t\in\Sigma}$ is still a family of holomorphic correspondences for every $n\ge1$.
\end{rmk}
 
\begin{lm} \label{deg_is_const}
    Let $\{f_t\}_{t\in\Sigma}$ be a family of holomorphic correspondences on $X$. Then for $0\leq q\leq k$, the dynamical degree $d_q(f_t)$ is constant as a function in $t$.
\end{lm}

\begin{proof}
    Let $k'=\dim\Sigma$, and let $\Gamma^n_t$ be the graph of $f^n_t$ in $X\times X$. By Lemma \ref{lemma:dyna-degree}, we have
    \[
        d_q(f_t)=\lim_{n\longrightarrow+\infty}(\lp f_t^*(\omega^q),\omega^{k-q}\rp)^{1/n}=\lim_{n\longrightarrow+\infty}\left(\int_{X\times X}[\Gamma^n_t]\wedge \pi_2^*(\omega^q)\wedge\pi_1^*(\omega^{k-q})\right)^{1/n}.
    \]
    So, it suffices to show that the quantity
    \[
        c_{q,n}(t):=\int_{X\times X}[\Gamma^n_t]\wedge \pi_2^*(\omega^q)\wedge\pi_1^*(\omega^{k-q})
    \]
    is constant in $t$. By Remark \ref{the_defs_are_equiv}, $\{f^n_t\}_{t\in\Sigma}$ is a family of holomorphic correspondences for every $n$. Denote by $\Gamma^n$ the graph of this family, which is a finite sum of irreducible analytic sets of dimension $k+k'$ on $\Sigma\times X\times X$. Define $\pi_\Sigma:\Sigma\times X\times X\to \Sigma$ by $\pi_\Sigma(t,x,y)=t$ and $\pi_{X\times X}:\Sigma\times X\times X\to X\times X$ by $\pi_{X\times X}(t,x,y)=(x,y)$. Then, by definition, we have $\lp [\Gamma^n],\pi_\Sigma,t\rp=[\Gamma^n_t]$ for every $t\in\Sigma$. Here, we identify $\{t\}\times X\times X$ with $X\times X$. Since $\pi_2^*(\omega^q)\wedge\pi_1^*(\omega^{k-q})$ is smooth, setting $R:=[\Gamma^n]\wedge\pi_{X\times X}^*\big(\pi_2^*(\omega^q)\wedge\pi_1^*(\omega^{k-q})\big)$ and $R_t:=[\Gamma^n_t]\wedge\pi_2^*(\omega^q)\wedge\pi_1^*(\omega^{k-q})$ we have $\lp R,\pi_\Sigma,t\rp=R_t$. By the theory of slicing (see \cite{federer}, and \cite[Section 2]{DS06AIF} for the case of positive closed currents), we have
    \begin{equation} \label{slicing_and_mass}
        \int_\Sigma \lp R_t, \mathbbm{1}\rp\Omega(t)=\lp R\wedge\pi_\Sigma^*(\Omega),\mathbbm1\rp=\lp(\pi_\Sigma)_*(R),\Omega\rp
    \end{equation}
    for every continuous form $\Omega$ of maximal degree with compact support in $\Sigma$, where $\mathbbm1$ is the function which is constantly equal to $1$. Observe that $(\pi_\Sigma)_*(R)$ is a current of degree $0$ on $\Sigma$, i.e. a function. From the fact that $[\Gamma^n]$, $\omega^q$ and $\omega^{k-q}$ are closed, it follows that $R$ is closed, and so also $(\pi_\Sigma)_*(R)$ is closed, hence constant since $\Sigma$ is connected. From this, the continuity of slicing, and equality \eqref{slicing_and_mass}, it follows that the quantity $\lp R_t,\mathbbm1\rp=c_{q,n}(t)$ is a constant function on $\Sigma$. This concludes the proof.
\end{proof}

We are going to prove that having small adjoint multiplicity is a generic property for holomorphic correspondences. The next result implies that for a family of holomorphic correspondences $\{f_t\}_{t\in\Sigma}$ on $X$, the condition about the local multiplicity $\rho(f_t)<\rho$ is true for generic $t$ as long as the set $\Sigma_\rho:=\{t\in\Sigma: \rho(f_t)<\rho\}$ is non-empty. When also $\{f_t^{-1}\}_{t\in\Sigma}$ is a family of holomorphic correspondences, since $\delta(f_t)=\rho(f_t^{-1})$, the same result holds for the adjoint multiplicity.

\begin{prop}\label{generic}
    Let $\Sigma$ be a complex manifold of dimension $k'$. Let $\{f_t\}_{t\in\Sigma}$ be a family of holomorphic correspondences on a compact K\"ahler manifold $X$. For every $\rho>0$ we have that $\Sigma_\rho$ is Zariski open in $\Sigma$.
\end{prop}

\begin{proof}
    Let $\Gamma$ be the graph of the family $\{f_t\}_{t\in\Sigma}$ in $\Sigma\times X\times X$. That is, $\Gamma=\{(t,x,y):y\in f_t(x)\}$. Let $\pi_1:\Sigma\times X\times X\rightarrow \Sigma\times X$ be the projection map $(t,x,y)\mapsto (t,x)$. Define 
    \begin{align*}
        W&=\{(t,x): \pi_1|_\Gamma\text{ has maximal multiplicity larger or equal to }\rho\text{ on }(t,x)\}\\
        &=\{(t,x): \text{the graph of } f_t \text{ has maximal multiplicity larger than or equal to }\rho \text{ at } x\}.
    \end{align*}
    Notice that $\pi_\Sigma(W)=\Sigma\setminus\Sigma_\rho$. Therefore, it suffices to prove that $W$ is an analytic set. Notice that this is a local property. Since $\pi_1|_\Gamma$ has finite fibres, it is a ramified covering. So, for every point $(t_0,x_0,y_0)\in \Gamma$ we can choose two neighbourhoods $U_0$ of $(t_0,x_0)$ and $V_0$ of $y_0$ such that $\pi_1|_{\Gamma\cap (U_0\times V_0)}$ is a ramified covering of degree $d$. Here $d$ is smaller than or equal to $d_0(f_t)$. We may choose $U_0$ and $V_0$ as local coordinate charts with local coordinates $(t^1,\ldots,t^{k'}; x^1,\ldots,x^k)$ and $(y^1,\dots,y^k)$. For every $a\in U_0$, denote by $a_1,\dots,a_d$ the preimages of $a$ under $\pi_1|_{\Gamma\cap (U_0\times V_0)}$, counted with multiplicity. Define functions
    \[
        \varphi_J(s,a):=\Bigg(\prod_{i=1}^k \prod_{\substack{j,h\in J,\\j\not=h}}\Big(s-\big(y^i(a_j)-y^i(a_h)\big)\Big)\Bigg)-s^{k\frac{\rho(\rho-1)}{2}}
    \]
    for every $J\subseteq \{1,\dots,d\}$, $|J|=\rho$, and $s\in\C$. Given $a\in U_0$, we have $\varphi_J(a,s)\equiv 0$ as a polynomial in $s$ if and only if we have $y^i(a_j)=y^i(a_h)$ for every $i=1,\dots,k$ and $j,h\in J$, which is equivalent to have $a_j=a_h$ for every $j,h\in J$.

    Consider now the function
    $$\varphi(s,a)=\prod_{\substack{J\subseteq \{1,\dots, d\},\\|J|=\rho}} \varphi_J(s,a),$$
    which is a polynomial in $s$ of degree at most $D=\dbinom{d}{\rho}\left(k\dfrac{\rho(\rho-1)}{2}-1\right)$ whose coefficients are symmetric analytic functions of $a_1,\dots,a_d$.

    Choose any set of distinct complex numbers $\{s_1,\dots,s_{D+1}\}$. Then $\varphi(s_l,a)$ is a holomorphic function on $U_0$ for every $l=1,\dots,D+1$. Let $V\big(\varphi(s_l,a)\big)$ be its zero set, which is an analytic set. We claim that 
    \[
        W=\bigcap_{l=1}^{D+1} V\big(\varphi(s_l,a)\big).
    \]
    In fact, for any $a$ belonging to the right hand side, $\varphi(s,a)$ is a polynomial of degree at most $D$ with $D+1$ different roots $s_1,\dots,s_{D+1}$. Therefore, it must be identically zero. Since $\varphi$ is the product of the $\varphi_J$'s, there exists a subset $J$ of $\{1,\dots,d\}$, with $|J|=\rho$, such that $\varphi_J(s,a)\equiv0$, which implies $a_j=a_h$ for every $j,h\in J$. It follows that $a\in W$.
    
    Conversely, if $a\in W$ we can find $1\le j_1<\dots<j_\rho\le d$ such that $a_{j_1}=\cdots=a_{j_\rho}$. Then the function $\varphi_J(s,a)$, $J=\{j_1,\dots,j_\rho\}$ is identically zero when seen as a polynomial in $s$, and so $\varphi(s_l,a)=0$ for every $l=1,\dots,D+1$. The proof is complete. 
\end{proof}

\subsection{Symmetric product and polynomial correspondences}

 Next, we study some families of holomorphic correspondences on the projective space defined by polynomials.

For $n>0$ and a holomorphic correspondence $f$ on $X$, we say that a $n$-tuple of points $(w_0,\dots,w_{n-1})$ in $X$ is a {\it chain of length $n$} if $w_i\in f(w_{i-1})$ for $1\leq i\leq n-1$. We say a chain is a {\it cycle} if furthermore $w_0\in f(w_{n-1})$. A point $w\in X$ is {\it periodic of period $n$} if there is a cycle of length $n$ starting with $w$. We have the following lemma, see for instance \cite[Lemma 4.3]{DKW20}.
    
    \begin{lm}\label{surface-delta}
        Let $h$ be a holomorphic correspondence on a compact Riemann surface $X$ whose graph contains no fibre of $\pi_1$ or $\pi_2$. Suppose the critical values of $h$ are not periodic. Then there exists an integer $D>0$ such that $\delta(h^n)\leq D$ for all $n$.
    \end{lm}

    As we see in the next example, a holomorphic correspondence on $\Pb^1$ satisfying the conditions in the above lemma does exist.

    \begin{ex}\label{bounded-delta}
         Fix positive integers $d_0\neq d_1$. For $c\in\C$, consider the correspondence $f_c$ which sends $z$ to $w$ on $\Pb^1=\C\cup\{\infty\}$ satisfying the equation $z^{d_1}(w^{d_0}+3w+1)=w^{d_0}+c$. Notice that $f_c^{-1}(\infty)=\{z\in\C: z^{d_1}=1\}$. Therefore, $\infty$ is not a critical value. Define $F_c=z^{d_1}(w^{d_0}+3w+1)-w^{d_0}-c$. By implicit function theorem, when $\partial_zF_c(z,w)=d_1z^{d_1-1}(w^{d_0}+3w+1)\neq 0$, the point $(z,w)$ is not a critical point. When $z=0$, we necessarily have $w^{d_0}+c=0$. Hence, the critical values belong to the set $\{w:w^{d_0}+c=0\}\cup\{w: w^{d_0}+3w+1=0\}$. From the next lemma it follows that $f_c$ satisfies the hypothesis of Lemma \ref{surface-delta} when $c$ is transcendental. 

         \begin{lm}\label{algebraic}
             If one of the points in $\{w:w^{d_0}+c=0\}\cup\{w: w^{d_0}+3w+1=0\}$ is periodic, then $c$ is algebraic. In particular, if we choose $c$ to be transcendental, critical values of $f_c$ are not periodic.
         \end{lm}
         \begin{proof}
             We first treat the case $d_0<d_1$.
             
             Take $w_0\in\{w:w^{d_0}+c=0\}$. We want to prove that $\infty$ does not belong to its forward orbit if $c$ is not algebraic. Suppose that it does. Then we get a sequence of points $(w_0,w_1,\dots,w_n)$ in $\Pb^1$, depending on $c$, such that $w_n=\infty$, which implies $w_{n-1}\in f_c^{-1}(\infty)$, and we have the following system of equations:
             \[
             \begin{cases}
                P_0(w_0,c):=w_0^{d_0} + c= 0, \\[6pt]
                P_1(w_0,w_1,c):=w_0^{d_1}\bigl(w_1^{d_0}+3w_1+1\bigr) - \bigl(w_1^{d_0}+c\bigr) = 0, \\[6pt]
                P_2(w_1,w_2,c):=w_1^{d_1}\bigl(w_2^{d_0}+3w_2+1\bigr) - \bigl(w_2^{d_0}+c\bigr) = 0, \\[6pt]
                \qquad\vdots \\[6pt]
                P_{n-1}(w_{n-2},w_{n-1},c):=w_{n-2}^{d_1}\bigl(w_{n-1}^{d_0}+3w_{n-1}+1\bigr) - \bigl(w_{n-1}^{d_0}+c\bigr) = 0, \\[6pt]
                P_n(w_{n-1}):=w_{n-1}^{d_1}-1=0.
            \end{cases}
            \]
            Viewing $P_n$ and $P_{n-1}$ as two polynomials of $w_{n-1}$ with coefficients in $\Z[w_{n-2},c]$, their resultant $\Res_{w_{n-1}}(P_n,P_{n-1})$ is an element in $\Z[w_{n-2},c]$. Moreover, $\Res_{w_{n-1}}(P_n,P_{n-1})=0$ if and only if they have a common root. Inductively, for a fixed $c\in\C$, the above system of polynomial equations has a solution $(w_0,\dots,w_{n-1})$ if and only if 
            \[
                P(c):=\Res_{w_0}(P_0,\Res_{w_1}(P_1,\dots,\Res_{w_{n-1}}(P_n,P_{n-1})))=0.
            \]
            This is a polynomial in $c$ with integer coefficients. If $P\not\equiv0$, this implies that $c$ is algebraic. If $P\equiv 0$, then for any $c\in\C$ the above system has a solution. We are going to show that this is not possible. Notice that $|w_0|=|c|^{1/d_0}$ for every $c$ and $|w_{n-1}|=1$. So it must be $n>1$, otherwise we will have $|c|=1$ for every $c\in\C$. Since $P_{n-1}(w_{n-2},w_{n-1},c)=0$, we can deduce that $|w_{n-2}|\sim |c|^{1/d_1}$ as $|c|\to\infty$. Then, since $d_0<d_1$ and $P_{n-2}(w_{n-3},w_{n-2},c)=0$, we have 
            \[
                |w_{n-3}|=\left|\frac{w_{n-2}^{d_0}+c}{w_{n-2}^{d_0}+3w_{n-2}+1}\right|^{1/d_1}\sim |c|^{(1-\frac{d_0}{d_1})\frac{1}{d_1}}.
            \]
            By induction, we deduce that $|w_0|\sim |c|^\alpha$ for some $0<\alpha\leq 1/d_1<1/d_0$, which contradicts with the fact that $|w_0|=|c|^{1/d_0}$.
            
            Suppose now that $w_0$ is periodic, and let $(w_0,\dots,w_{n-1})$ be a cycle in $\Pb^1$. From the above discussion, none of the points $(w_0,\dots,w_{n-1})$ is $\infty$. Therefore, we have the following system of polynomial equations:
            \[
             \begin{cases}
                Q_0(w_0,c):=w_0^{d_0} + c= 0, \\[6pt]
                Q_1(w_0,w_1,c):=w_0^{d_1}\bigl(w_1^{d_0}+3w_1+1\bigr) - \bigl(w_1^{d_0}+c\bigr) = 0, \\[6pt]
                Q_2(w_1,w_2,c):=w_1^{d_1}\bigl(w_2^{d_0}+3w_2+1\bigr) - \bigl(w_2^{d_0}+c\bigr) = 0, \\[6pt]
                \qquad\vdots \\[6pt]
                Q_{n-1}(w_{n-2},w_{n-1},c):=w_{n-2}^{d_1}\bigl(w_{n-1}^{d_0}+3w_{n-1}+1\bigr) - \bigl(w_{n-1}^{d_0}+c\bigr) = 0, \\[6pt]
                Q_{n}(w_{n-1},w_0,c):=w_{n-1}^{d_1}\bigl(w_0^{d_0}+3w_0+1\bigr) - \bigl(w_0^{d_0}+c\bigr) =0.
            \end{cases}
            \]
            In the same way as above, we can prove that $c$ must be a solution of some nonzero polynomial with integer coefficients.

            \smallskip
            
            Take now $w_0\in\{w: w^{d_0}+3w+1=0\}$. For this case, we need to consider $c$ outside of the set of algebraic numbers $\{3\omega+1\mid \omega^{d_0}+3\omega+1=0\}$. This implies that $w^{d_0}+3w+1$ and $w^{d_0}+c$ have no common zeroes. Therefore, the only point $z$ which can be sent to $w_0$ is $\infty$. So, to prove that $w_0$ does not belong to a periodic cycle, we just need to show that $\infty$ does not belong to its forward orbit. Suppose by contradiction that it does. Then, as in the previous case, we have a system of equations:
            $$\begin{cases}
                R_0(w_0,c):=w_0^{d_0} +3w_0+1= 0, \\[6pt]
                R_1(w_0,w_1,c):=w_0^{d_1}\bigl(w_1^{d_0}+3w_1+1\bigr) - \bigl(w_1^{d_0}+c\bigr) = 0, \\[6pt]
                R_2(w_1,w_2,c):=w_1^{d_1}\bigl(w_2^{d_0}+3w_2+1\bigr) - \bigl(w_2^{d_0}+c\bigr) = 0, \\[6pt]
                \qquad\vdots \\[6pt]
                R_{n-1}(w_{n-2},w_{n-1},c):=w_{n-2}^{d_1}\bigl(w_{n-1}^{d_0}+3w_{n-1}+1\bigr) - \bigl(w_{n-1}^{d_0}+c\bigr) = 0, \\[6pt]
                R_n(w_{n-1}):=w_{n-1}^{d_1}-1=0.
            \end{cases}$$
            Except when $n=1$, we can repeat the same proof of the previous case and obtain $|w_0|\sim|c|^\alpha$ for some $\alpha>0$, contradicting the fact that $|w_0|\sim 1$. The case $n=1$ cannot happen. Otherwise, $w_0$ would be a common zero of $w^{d_0}+3w+1$ and $w^{d_1}-1$. The equation $w_0^{d_1}-1=0$ would imply $|w_0|=1$, but this in turn would imply $0=|w_0^{d_0}+3w_0+1|\ge |3w_0|-|w_0^{d_0}|-1=1$, which is a contradiction.

            \medskip
            We now treat the case $d_1>d_0$.

            For the case $w_0\in\{w:w^{d_0}+c=0\}$, the same proof as above yields $|w_0|\sim |c|^{1/d_1}$ or $|w_0|\sim1$ based on the parity of $n$. Both give a contradiction to $|w_0|=|c|^{1/d_0}$.
            
            For the case $w_0\in\{w: w^{d_0}+3w+1=0\}$, again we have $|w_0|\sim |c|^{1/d_1}$ or $|w_0|\sim1$ based on the parity of $n$. By construction, $w_0$ must be one of $d_0$ constant values of modulus different from $1$. If $c$ is sufficiently large, this rules out $|w_0|\sim |c|^{1/d_1}$. If the parity of $n$ gives $|w_0|\sim1$, we can choose $c$ very large, so that the reasoning of the proof tells us that $|w_0|$ is arbitrarily close to $1$, which is not possible by construction. This completes the proof.
         \end{proof}
    \end{ex}
    
    Let $h$ be a holomorphic correspondence on $\Pb^1$ of bidegree $(d_0,d_1)$. Define $\widehat{f}=(h,\dots,h)$ to be a holomorphic correspondence on $(\Pb^1)^k$. Since $H^{0,1}(\Pb^1,\C)=H^{1,0}(\Pb^1,\C)=0$, by K\"unneth formula we have 
    \[
        H^{q,q}((\Pb^1)^k,\C)=\bigoplus_{\substack{(s_1,s_2,\dots,s_k)\in\{0,1\}^k\\
        s_1+s_2+\cdots+s_k=q}}H^{s_1,s_1}(\Pb^1,\C)\otimes\cdots\otimes H^{s_k,s_k}(\Pb^1,\C).
    \]
    Notice that the action $h^*$ on $H^{0,0}(\Pb^1,\C)$ (resp. $H^{1,1}(\Pb^1,\C)$) is multiplication by $d_0$ (resp. $d_1$). Therefore, the action $\widehat{f}^*$ on $H^{q,q}((\Pb^1)^k,\C)$ is given by the diagonal matrix $d_1^qd_0^{k-q}\cdot I$. In particular, we have $d_q(\widehat{f})=d_1^qd_0^{k-q}$.

    By the proof of \cite[Lemma 5.4.5]{DS09AM} and \cite[Section 1]{GHK23CGD}, there exists a holomorphic covering map $\pi$ of degree $k!$ from $(\Pb^1)^k$ to $\Pb^k$ given by $\pi([x_1:y_1],\dots,[x_k:y_k])=[\eta_0:\cdots:\eta_k]$ where $\eta_j$ is given by the formula 
      \[
          \eta_j([x_1:y_1],\dots,[x_k:y_k]):=\sum_{\substack{(i_1,i_2,\dots,i_k)\in\{0,1\}^k\\
        i_1+i_2+\cdots+i_k=k-j} }\prod_{l=1}^k x_l^{i_l}\cdot y_l^{1-i_l}.
      \]
     Define
    \[
        f(\pi(z_1,\dots,z_k)):=\{\pi(w_1,\dots,w_k):w_i\in h(z_i) \text{ for }1\leq i\leq k\}.
    \]
    Since $\pi$ is surjective, $f$ is a well-defined multivalued map on $\Pb^k$ and its graph is given by $\Gamma_{f}:=(\pi\times \pi)(\Gamma_{\widehat{f}})$. We call $f$ the {\it $k$-fold symmetric product of $h$}. Since $\Gamma_{\widehat{f}}$ is a union of irreducible analytic sets of dimension $k$ in $(\Pb^1)^k\times (\Pb^1)^k$, so is $\Gamma_{f}$ in $\Pb^k\times\Pb^k$. By an abuse of notation, we use $\pi_1$ and $\pi_2$ to denote the canonical projections from $(\Pb^1)^k\times (\Pb^1)^k$ or $\Pb^k\times\Pb^k$ to their factors. Then we have $\pi_i\circ (\pi\times \pi)=\pi\circ\pi_i$ for $i=1,2$. The next lemma is easy to verify.
    
    \begin{lm}\label{indeterminacy-set}
        For any $Z\in \Pb^k$, we have $\pi_i^{-1}(Z)\cap\Gamma_{f}=\bigcup_{w\in\pi^{-1}(Z)}(\pi\times \pi)(\pi_i^{-1}(w)\cap\Gamma_{\widehat{f}})$ for $i=1,2$.
    \end{lm}
    
    The following is a generalization of \cite[Proposition 1.2]{GHK23CGD}.
      
      \begin{prop}\label{sym-prod}
          The following properties are true:
          \begin{nlist}
              \item $f$ is a holomorphic correspondence on $\Pb^k$ such that $f^N\circ \pi=\pi\circ \widehat{f}^N$ for all $N\in\N$ and its dynamical degrees are given by $d_q(f)=d_1^qd_0^{k-q}$;
              \item If $\widehat{f}^{-1}$ is a holomorphic correspondence, then so is $f^{-1}$. In this case, the adjoint multiplicity of $f^N$ satisfies $\delta(f^N)\leq k!\cdot\delta(\widehat{f}^N)$ for all $N$;
              \item If $h$ is given by $\Gamma_h=\{(z,w)\in \Pb^1\times \Pb^1:g_0(z)=g_1(w)\}$ where $g_i$ is a holomorphic endomorphism of $\Pb^1$ of algebraic degree $d_i$ for $i=0$ and $1$, then $f$ is given by $\Gamma_{f}=\{(Z,W)\in\Pb^k\times \Pb^k: G_0(Z)=G_1(W)\}$ where $G_i$ is a holomorphic endomorphism of $\Pb^k$ of algebraic degree $d_i$ for $i=0$ and $1$.
          \end{nlist}
      \end{prop}
      \vspace{-0.3\baselineskip}
      
      \begin{proof}
            We first prove (i) and (ii). Lemma \ref{indeterminacy-set} implies that when $\widehat{f}$ (resp. $\widehat{f}^{-1}$) is a holomorphic correspondence, so is $f$ (resp. $f^{-1}$). By definition, we have $f\circ\pi=\pi\circ \widehat{f}$ and by induction, we have $f^N\circ\pi=\pi\circ \widehat{f}^N$ for all $N\in\N$. Since we have $\widehat{f}^*\circ \pi^*=\pi^*\circ f^*$ and $f^*$ on $H^{q,q}(\Pb^k,\C)$ is the multiplication by $d_q(f)$, we deduce that $d_q(f)=d_q(\widehat{f})=d_1^qd_0^{k-q}$. The last assertion of (ii) follows from the fact that $\pi$ is a ramified covering of degree $k!$.

            It remains to prove (iii). By \cite[Proposition 1.2]{GHK23CGD}, there exist holomorphic endomorphisms $G_0$ and $G_1$ of $\Pb^k$ such that 
            \[
                \pi\circ (g_i,\dots,g_i)=G_i\circ \pi \quad\text{for } i=0,1.
            \]
            For any $(\pi(z_1,\dots,z_k),\pi(w_1,\dots,w_k))\in\Gamma_{f}$, we have 
            \begin{align*}
                G_0(\pi(z_1,\dots,z_k))&= \pi(g_0(z_1),\dots,g_0(z_k)) \\
                    &= \pi(g_1(w_1),\dots,g_1(w_k))= G_1(\pi(w_1,\dots,w_k))
            \end{align*}
            where the second equality comes from the definition of $h$. Conversely, consider any $(Z,W)$ satisfying the equation $G_0(Z)=G_1(W)$. Since $\pi$ is surjective, we may assume $Z=\pi(z_1,\dots,z_k)$ and $W=\pi(w_1,\dots,w_k)$. It then follows that 
            \[
                \pi(g_0(z_1),\dots,g_0(z_k))=\pi(g_1(w_1),\dots,g_1(w_k)).
            \]
            By changing the order of $(w_1,\dots,w_k)$, we get that $g_0(z_i)=g_1(w_i)$ for $1\leq i\leq k$. Therefore, $(Z,W)\in\Gamma_{f}$. This concludes the proof.
      \end{proof}

      Recall that the family of holomorphic endomorphisms of degree $d_0$ denoted by $\mathcal H_{d_0}(\Pb^k)$ can be parametrized by a Zariski dense open set in $\Pb^{N_{k,d_0}-1}$ where $N_{k,d_0}=(k+1)(d_0+k)!/(d_0!k!)$. For two integers $d_1,d_0\geq 1$, define a set of holomorphic correspondences on $\Pb^k$ by 
      \begin{align*}
          \mathcal F_{(d_0,d_1)}:=\Big\{f:\Gamma_f=\{(Z&,W)\in\Pb^k\times \Pb^k: G_0(Z)=G_1(W)\}\\
              &\text{ where } G_0\in\mathcal H_{d_0}(\Pb^k) \text{ and } G_1\in\mathcal{H}_{d_1}(\Pb^k)\Big\}.
      \end{align*}
      Consider the set $\{(G_0,G_1,Z,W)\in\mathcal{H}_{d_0}(\Pb^k)\times \mathcal{H}_{d_1}(\Pb^k)\times\Pb^k\times\Pb^k: G_0(Z)=G_1(W)\}$. Then it is an analytic set and therefore, $\mathcal{F}_{(d_0,d_1)}$ is a family of holomorphic correspondences. Define $\mathcal F_{(d_0,d_1)}^{-1}:=\{f^{-1}: f\in \mathcal F_{(d_0,d_1)}\}$. Notice that this is just $\mathcal{F}_{(d_1,d_0)}$.

      \begin{cor}\label{pk-corr}
          For $0< q\leq k$ and $d_1>d_0$, there exists an integer $N$ such that for generic elements $f$ in $\mathcal F_{(d_0,d_1)}$, $f^N$ has $q$-small adjoint multiplicity.
      \end{cor}

\begin{proof}
    Take $h_0$ as in Example \ref{bounded-delta}. Then by Lemma \ref{surface-delta} and Proposition \ref{sym-prod}, its $k$-fold symmetric product $f_0$ belongs to $\mathcal{F}_{(d_0,d_1)}$ and 
    \[
        \delta(f_0^n)\leq k!\delta(\widehat{f}_0^n)\leq k!D^k
    \]
    for any $n\geq 1$. Notice that for any $f\in\mathcal F_{(d_0,d_1)}$ and $n\geq 1$, $d_q(f^n)/d_{q-1}(f^n)=(d_1/d_0)^n$. Choose an integer $N$ such that $25k(4k!D^k)^{k-q+1}<(d_1/d_0)^N$. 
    Apply Proposition \ref{generic} to the family $\{f^{-N}\}_{f\in\mathcal F_{(d_0,d_1)}}$ and $\rho=\frac{1}{4}[\frac{1}{25k}(\frac{d_1}{d_0})^N]^{1/(k-q+1)}$. Notice that $\Sigma_\rho$ is exactly the set of $f\in\mathcal F_{(d_0,d_1)}$ such that $f^N$ has $q$-small adjoint multiplicity. Then we have $f_0\in\Sigma_\rho$, and we deduce that $\Sigma_\rho$ is a Zariski dense open set, which is the desired result.
\end{proof}

\begin{cor}\label{poly-small}
    Consider the family $\mathcal F_{(d_0,d_1)}$ of correspondence on $\Pb^k$ as in Corollary \ref{pk-corr}. Then there exists an integer $N$ such that for generic elements $f$ in $\mathcal F_{(d_0,d_1)}$, $f^N$ has small multiplicity.
\end{cor}

\begin{proof}
    By definition, for $f\in\mathcal F_{(d_0,d_1)}$, the fact that $f^N$ has small multiplicity means that we have $50k(4\delta(F^N))^{k+1}<(d_1/d_0)^N$, or more precisely $\delta(f^N)\cdot\delta(f^{-N})<\frac{1}{4}[\frac{1}{50k}(\frac{d_1}{d_0})^N]^{1/(k+1)}$. By Example \ref{bounded-delta}, there are $f_0\in\mathcal F_{(d_0,d_1)}$ and $f_1\in\mathcal F_{(d_1,d_0)}$ such that both $\delta(f_0^n)$ and $\delta(f_1^n)$ are bounded from above by a constant independent of $n$. Now we choose $N$ large enough such that 
    \[
        \max\{\delta(f_0^N), \delta(f_1^N)\}<\rho_0:=\frac{1}{2}\left[\frac{1}{50k}\left(\frac{d_1}{d_0}\right)^N\right]^{1/2(k+1)}.
    \]
    Applying Proposition \ref{generic} to the families $\mathcal{F}_{(d_0,d_1)}^N$ and $\mathcal{F}_{(d_1,d_0)}^N$ with $\rho=\rho_0$, we deduce that for generic $f\in\mathcal{F}_{(d_0,d_1)}$, we have $\max\{\delta(f^N), \delta(f^{-N})\}<\rho_0$, which implies the desired result.
\end{proof}

\appendix
\section{Properties of the pullback operator} \label{appendix}

Let $f$ be a holomorphic correspondence on a compact K\"ahler manifold $X$ with graph $\Gamma$. We do not make any additional assumption on $f$. Recall that the pullback operator
$$f^*(T)=(\pi_1)_*(\pi_2^*(T)\wedge[\Gamma])$$
is well defined when the current $T$ is smooth. The action of $f^*$ has been studied in detail in \cite{DNV18AM}. However, the definitions of $\cC_q^c$, $\cD_q^c$ and $\cD_q^{0c}$ given in \cite{DNV18AM} are more restrictive than ours: the authors only consider currents which can be written as a \emph{difference of two positive closed currents} with continuous super-potentials. This condition is not satisfied by all currents with continuous super-potentials. Therefore, the proofs in \cite{DNV18AM} do not cover our case. We give here an adapted version of these proofs for the properties of $f^*$ that we need. The main difference is the way that the pullback is constructed.

\begin{prop}
    The operator $f^*$ can be extended to a linear operator from $\cD_q^c$ to itself for $0\leq q\leq k$. This extension is continuous in the following sense: if the sequence $T_n$ converges SP-uniformly to $T$, then $f^*(T_n)$ converges to $f^*(T)$ in the sense of currents. It also preserves $\cC_q^c$ and $\cD_q^{0c}$. For every $T\in \cD_q^c$, we have $f^*\{T\}=\{f^*(T)\}$, and $(f^n)^*=(f^*)^n$ on both $\cD_q^c$ and $H^{q,q}(X,\R)$. Moreover, the action of $f^*$ on $H^{q,q}(X,\R)$ is dual to the one of $f_*$ on $H^{k-q,k-q}(X,\R)$.
\end{prop}

\begin{proof}
    Take $T\in\cD_q^c$. Since $\pi_2$ is a submersion, the current $\pi_2^*(T)$ is well defined. We want to study its super-potential. We can normalize potentials in $X\times X$ in such a way that, if $S$ is an exact smooth $(2k-q+1,2k-q+1)$-form on $X\times X$ with normalized potential $U_S$, then $(\pi_2)_*(U_S)$ is a normalized potential of $(\pi_2)_*(S)$ on $X$. Hence, we get
    \begin{equation} \label{projecting_the_superpot}
        \cU_{\pi_2^*(T)}(S)=\lp\pi_2^*(T),U_S\rp=\lp T,(\pi_2)_*(U_S)\rp=\cU_T((\pi_2)_*(S)).
    \end{equation}
    Since the pushforward $(\pi_2)_*$ acts continuously on currents and $T$ has a continuous super-potential, it follows that the super-potential of $\pi_2^*(T)$ is continuous. Therefore, the intersection $\pi_2^*(T)\wedge[\Gamma]$ is well defined (see \cite[Section 3.3]{DS10JAG}), and so is the pullback $f^*(T)$. We need to prove that it has a continuous super-potential.
    
    We prove it first for currents $T$ such that $\{T\}=0$. Given an exact smooth $(k-q+1,k-q+1)$-form $R$ on $X$, we have
    $$\cU_{f^*(T)}(R)=\lp f^*(T),U_R\rp=\lp\pi_2^*(T)\wedge[\Gamma],\pi_1^*(U_R)\rp.$$
    By \cite[Theorem 2.1]{DNV18AM}, we can find a sequence of smooth closed forms $([\Gamma]_j)$ converging to $[\Gamma]$ in $\cD_k(X\times X)$. Since the wedge product is a continuous operation, we get
    \begin{equation} \label{pb_cont_1}
        \cU_{f^*(T)}(R)=\lim_{j\rightarrow\infty}\lp\pi_2^*(T)\wedge[\Gamma]_j,\pi_1^*(U_R)\rp=\lim_{j\rightarrow\infty}\lp T,(\pi_2)_*([\Gamma]_j\wedge\pi_1^*(U_R))\rp.
    \end{equation}
    Since $\{T\}=0$, we have
    \begin{equation} \label{pb_cont_2}
        \lp T,(\pi_2)_*([\Gamma]_j\wedge\pi_1^*(U_R))\rp=\cU_T(\ddc(\pi_2)_*([\Gamma]_j\wedge\pi_1^*(U_R)))=\cU_T((\pi_2)_*([\Gamma]_j\wedge\pi_1^*(R))).
    \end{equation}
    Putting \eqref{pb_cont_1} and \eqref{pb_cont_2} together, we obtain
    \begin{equation}\label{pb_cont_3}
        \cU_{f^*(T)}(R)=\lim_{j\rightarrow\infty}\cU_T((\pi_2)_*([\Gamma]_j\wedge\pi_1^*(R)))=\cU_T(f_*(R)),
    \end{equation}
    where in the last equality we used the continuity of $\cU_T$. Since we also have that $f_*$ acts continuously on currents, \eqref{pb_cont_3} extends to all currents $R\in\cD_{k-q+1}^0$. Therefore, $\cU_{f^*(T)}$ is continuous.

    Take now any current $T\in\cD_q^c$. We can write it as the sum of a smooth current $\alpha_T$ and an exact current $T'$, which will still have a continuous super-potential $\cU_{T'}=\cU_T-\cU_{\alpha_T}$. By linearity of the pullback and of super-potentials, we have $\cU_{f^*(T)}=\cU_{f^*(T')}+\cU_{f^*(\alpha_T)}$. We just proved the continuity of $\cU_{f^*(T')}$. Since we can write $\alpha_T$ as the difference of two positive smooth forms, the continuity of $\cU_{f^*(\alpha_T)}$ follows from \cite[Lemma 4.3]{DNV18AM}. This proves the first assertion of the proposition.

    We now prove the continuity of $f^*$ acting on $\cD_q^c$. Take $T_n,T\in\cD_q^c$ such that the sequence $T_n$ converges SP-uniformly to $T$. From \eqref{projecting_the_superpot} it follows that the sequence $\pi_2^*(T_n)$ converges SP-uniformly to $\pi_2^*(T)$. The continuity of $f^*$ then follows from \cite[Theorem 3.3.2]{DS10JAG}.

    Since all operations in the definition of $f^*$ preserve positivity, $f^*$ preserves $\cC_q^c$.  Since $[\Gamma]$ is closed, $f^*$ commutes with $\de$ and $\debar$. Therefore, for every $T\in D_q^c$ we have $f^*\{T\}=f^*\{T\}$, and $f^*$ preserves $D_q^{0c}$.
    
    By \cite[Lemma 4.6]{DNV18AM}, we have $(f^n)^*=(f^*)^n$ when acting on differences of positive closed currents with continuous super-potentials. By the continuity that we proved above, this equality extends to all currents in $\cD_q^c$, and the same equality holds for the induced maps in cohomology.
    
    Finally, the duality of $f^*$ and $f_*$ in cohomology follows directly from their definitions when evaluated on smooth forms. The proof is complete.
\end{proof}

\bibliographystyle{plain}
\bibliography{refs}

\end{document}